\documentclass[oneside]{amsart}
\usepackage{silence}
\usepackage[all]{xy}
\usepackage{enumerate}
\usepackage{marginnote}
\usepackage{subcaption,nicefrac}
\usepackage[unicode,naturalnames,hidelinks]{hyperref}
\usepackage[utf8]{inputenc}
\usepackage[T1]{fontenc}
\usepackage{microtype}
\usepackage[backref=true,maxalphanames=6,isbn=false,url=false,style=alphabetic,maxbibnames=6,backend=biber,safeinputenc]{biblatex}
\DeclareFieldFormat{postnote}{#1}
\DeclareFieldFormat{multipostnote}{#1}
\bibliography{1131}

\usepackage{stix}

\DeclareMathOperator{\mypart}{\mathsf{COB}}
\DeclareMathOperator{\mylin}{\mathsf{LCU}}
\DeclareMathOperator{\prob}{pr}
\newcommand{\rel}{{\mathrm{rel}}}
\newcommand{\defeq}{\coloneq}
\renewcommand{\eqdef}{\eqcolon}

\newtheorem{theorem}[equation]{Theorem}

\newtheorem{construction}[equation]{Lemma/Construction}

\newtheorem{lemma}[equation]{Lemma}

\theoremstyle{definition}
\newtheorem{definition}[equation]{Definition}
\theoremstyle{remark}
\newtheorem{remark}[equation]{Remark}
\newtheorem*{notation*}{Notation}
\newtheorem{assumption}[equation]{Assumption}

\newtheorem{fact}[equation]{Fact}
\newtheorem{facts}[equation]{Facts}

\numberwithin{equation}{section}

\newcommand{\forces}{\Vdash}

\DeclareMathOperator{\dom}{dom}

\newcommand{\mye}{*+[F.]{\phantom{\lambda}}}
\newcommand{\myex}{*+[F.]{\phantom{x}}}

\DeclareMathOperator{\Av}{Av}
\DeclareMathOperator{\Leb}{Leb}

\newcommand{\QII}{\tilde{\mathbb{E}}}
\newcommand{\pdfQII}{Q2}
\newcommand{\myparti}{\mypart_i}
\newcommand{\mypartI}{\mypart_1}
\newcommand{\mypartII}{\mypart_2}
\newcommand{\mypartIII}{\mypart_3}
\newcommand{\mypartIV}{\mypart_4}
\newcommand{\mylini}{\mylin_i}
\newcommand{\mylinI}{\mylin_1}
\newcommand{\mylinII}{\mylin_2}
\newcommand{\mylinIII}{\mylin_3}
\newcommand{\mylinIV}{\mylin_4}
\DeclareMathOperator{\cf}{cf}
\DeclareMathOperator{\supp}{supp}
\DeclareMathOperator{\Rel}{R}
\newcommand{\Ri}{\Rel_i}
\DeclareMathOperator{\RI}{R_1}
\DeclareMathOperator{\RII}{R_2}
\DeclareMathOperator{\RIII}{R_3}
\DeclareMathOperator{\RIV}{R_4}

\newcommand{\Pa}{\mathbb{P}^5}

\newcommand{\PaVI}{\mathbb{P}^6}
\newcommand{\PaVII}{\mathbb{P}^7}
\newcommand{\PaVIII}{\mathbb{P}^8}
\newcommand{\PaIX}{\mathbb{P}^9}

\DeclareMathOperator{\cov}{cov}

\DeclareMathOperator{\stem}{trunk}

\DeclareMathOperator{\loss}{loss}

\DeclareMathOperator{\cof}{cof}
\DeclareMathOperator{\non}{non}
\DeclareMathOperator{\add}{add}

\newcommand{\covnull}{\cov(\mathcal N)}
\newcommand{\cofnull}{\cof(\mathcal N)}
\newcommand{\addnull}{\add(\mathcal N)}
\newcommand{\nonnull}{\non(\mathcal N)}
\newcommand{\covmeager}{\cov(\mathcal M)}
\newcommand{\cofmeager}{\cof(\mathcal M)}
\newcommand{\addmeager}{\add(\mathcal M)}
\newcommand{\nonmeager}{\non(\mathcal M)}

\subjclass[2010]{03E17}
\keywords{Set theory of the reals, Cicho\'n's  diagram, Forcing, Compact cardinals}
\date{2018-09-28}
\title{Another ordering of the ten cardinal characteristics in Cicho\'n's diagram}
\dedicatory{Dedicated to the memory of Bohuslav Balcar (1943--2017)}
\thanks{Supported by 
Austrian Science Fund (FWF): 
P26737 \& P30666 (first author),
European Research Council grant ERC-2013-ADG 338821 (second author).
The third author is recipient of a DOC Fellowship of the Austrian Academy of Sciences at the Institute of Discrete Mathematics and Geometry, TU Wien.
This is publication number 1131 of the second author.}
\author{Jakob Kellner}
\address{Technische Universität Wien (TU Wien).}
\email{jakob.kellner@tuwien.ac.at}
\urladdr{\url{http://dmg.tuwien.ac.at/kellner/}}
\author{Saharon Shelah}
\address{The Hebrew University of Jerusalem and Rutgers University.}
\email{shlhetal@mat.huji.ac.il}
\urladdr{\url{http://shelah.logic.at/}}
\author{Anda Ramona T{\u{a}}nasie}
\address{Technische Universität Wien (TU Wien).}
\email{anda-ramona.tanasie@tuwien.ac.at}
\begin{document}

\begin{abstract}
It is consistent that 
\[
\aleph_1 < \addnull <  \addmeager= \mathfrak{b} < \covnull <
\nonmeager < \covmeager = 2^{\aleph_0}.\]
Assuming  four strongly compact cardinals, it is consistent that 
\begin{multline*}
\aleph_1 < \addnull <\addmeager=\mathfrak{b} < \covnull < 
\nonmeager < \\
<
\covmeager < \nonnull < 
\cofmeager= \mathfrak{d} < \cofnull < 2^{\aleph_0}.
\end{multline*}
\end{abstract}

\maketitle
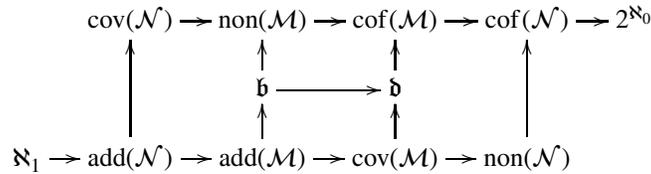
\begin{figure}
\[
\xymatrix@=2.5ex{
&            \covnull\ar[r]        & \nonmeager \ar[r]      &  \cofmeager \ar[r]     & \cofnull\ar[r] &   2^{\aleph_0}\\
&                               & \mathfrak b\ar[r]\ar[u]  &  \mathfrak d\ar[u] &              \\ 
\aleph_1\ar[r]   &\addnull\ar[r]\ar[uu] & \addmeager\ar[r]\ar[u] &  \covmeager\ar[r]\ar[u]& \nonnull\ar[uu] 
}
\]    
    \caption{\label{fig:Cichon}Cicho\'n's  diagram}
\end{figure}

\section*{Introduction}
We assume that the reader is familiar with basic properties of Amoeba, Hechler, random and Cohen forcing, and with the cardinal characteristics in Cicho\'n's  diagram, given in Figure~\ref{fig:Cichon}:
An arrow between $\mathfrak x$ and $\mathfrak y$ indicates that ZFC proves
$\mathfrak x\le \mathfrak y$. 
Moreover, $\max(\mathfrak d,\nonmeager)=\cofmeager$ and $\min(\mathfrak b,\covmeager)=\addmeager$. 
These (in)equalities are the only one
provable. More precisely, all assignments of the values $\aleph_1$ and $\aleph_2$
to the characteristics in Cicho\'n's diagram are consistent, provided they do
not contradict the above (in)equalities.  (A complete proof can be found
in~\cite[ch.~7]{BJ}.)

In the following, we will only deal with
the ten ``independent'' characteristics listed in Figure~\ref{fig:ten} (they 
determine $\cofmeager$ and $\addmeager$).

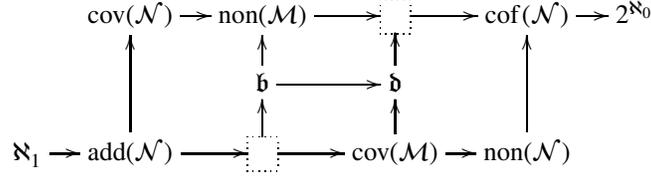
\begin{figure}
\[
\xymatrix@=2.5ex{
&            \covnull\ar[r]        & \nonmeager \ar[r]      &  \mye \ar[r]     & \cofnull\ar[r] &   2^{\aleph_0}\\
&                               & \mathfrak b\ar[r]\ar[u]  &  \mathfrak d\ar[u] &              \\ 
\aleph_1\ar[r]   &\addnull\ar[r]\ar[uu] & \mye\ar[r]\ar[u] &  \covmeager\ar[r]\ar[u]& \nonnull\ar[uu] 
}
\]    
    \caption{\label{fig:ten}The ten ``independent'' characteristics.}
\end{figure}

\begin{figure}
  \centering
  \begin{minipage}[b]{0.49\textwidth}
\[
\xymatrix@=2.5ex{
&            \lambda_2\ar[r]        & \lambda_4 \ar[r]      &  \mye \ar[r]     & \lambda_8\ar[r] &\lambda_9 \\
&                               & \lambda_3\ar[r]\ar[u]  &  \lambda_6\ar[u] &              \\ 
  \aleph_1\ar[r] & \lambda_1\ar[r]\ar[uu] & \mye\ar[r]\ar[u] &  \lambda_5\ar[r]\ar[u]& \lambda_7\ar[uu]  
}
\]       
    \caption{\label{fig:oldcards}The old order.}
  \end{minipage}
  \hfill
  \begin{minipage}[b]{0.49\textwidth}
\[
\xymatrix@=2.5ex{
&            \lambda_3\ar[r]        & \lambda_4 \ar[r]      &  \mye \ar[r]     & \lambda_8\ar[r] &\lambda_9 \\
&                               & \lambda_2\ar[r]\ar[u]  &  \lambda_7\ar[u] &              \\ 
  \aleph_1\ar[r] & \lambda_1\ar[r]\ar[uu] & \mye\ar[r]\ar[u] &  \lambda_5\ar[r]\ar[u]& \lambda_6\ar[uu]  
}
\]       
    \caption{\label{fig:ourcards}The new order.}
  \end{minipage}
\end{figure}

\pagebreak
\textbf{Regarding the left hand side,} it was shown in~\cite{MR3513558} that consistently 
\begin{equation}\tag{$\text{left}_\text{old}$}\label{eq:leftold}
\aleph_1  < \addnull  < \covnull <  \addmeager= \mathfrak{b} <
\nonmeager <\covmeager = 2^{\aleph_0}.
\end{equation}
(This corresponds to $\lambda_1$ to $\lambda_5$ in Figure~\ref{fig:oldcards}.)
The proof is repeated in~\cite{ten}, in a slightly different form which is more convenient
for our purpose. Let us call this construction the ``old construction''. 

In this paper, building on~\cite{sh}, we give a construction to get a different order for these characteristics, 
where we swap $\covnull$ and $\mathfrak{b}$:
\begin{equation}\tag{$\text{left}_\text{new}$}\label{eq:leftnew}
\aleph_1 <   \addnull  < \addmeager= \mathfrak{b} < \covnull <
\nonmeager < \covmeager = 2^{\aleph_0}.
\end{equation}
(This corresponds to $\lambda_1$ to $\lambda_5$ in Figure~\ref{fig:ourcards}.)

This construction is more complicated than the old one. Let us briefly describe the reason:
In both constructions, we assign 
to each of the
cardinal characteristics of the left hand side
a relation $\Rel$. E.g., we use the ``eventually different'' relation
$\RIV\subseteq \omega^\omega\times \omega^\omega$ 
for $\nonmeager$.
We can then show that the characteristic
remains ``small'' (i.e., is at most the intended value $\lambda$ in the final model), because all single forcings we use in the iterations are either small
(i.e., smaller than $\lambda$) or are ``$\Rel$-good''. However,
 $\mathfrak b$ (with the ``eventually dominating'' relation $\RII\subseteq \omega^\omega\times \omega^\omega$) 
is an exception: We do not know any variant of an
eventually different forcing (which we need to increase $\nonmeager$) which satisfies that all of its subalgebras are $\RII$-good. Accordingly, the main effort (in both constructions) is to show that  $\mathfrak b$ remains small.

In the old construction,
each non-small forcing is a ($\sigma$-centered) subalgebra of 
the eventually different forcing $\mathbb E$. To deal with 
such forcings,  
ultrafilter limits of sequences of $\mathbb E$-conditions are introduced and used (and we require that all $\mathbb E$-subforcings are basically
$\mathbb E$ intersected with some 
model, and thus closed 
under limits of sequences in the model). 
In the new construction, we have to deal with an additional kind of ``large'' forcing: (subforcings of) random forcing.
Ultrafilter limits do not work any more, but, similarly to~\cite{sh}, we can use 
finite additive measures (FAMs) and interval-FAM-limits of random conditions.
But now  $\mathbb E$ doesn't seem to work with interval-FAM-limits any more,
so we replace it with a creature forcing notion $\QII$.

We also have to show that $\covnull$ remains small. In the old construction, we could use a rather simple (and well understood) relation $\Rel^\mathrm{old}$ and 
use the fact that all $\sigma$-centered forcings are $\Rel^\mathrm{old}$-good:
As all large forcings are subalgebras of either eventually different forcing or of Hechler forcing, they are all $\sigma$-centered.
In the new construction, the large forcings we have to deal with are subforcings of  
$\QII$. But $\QII$ is not $\sigma$-centered,
just $(\rho,\pi)$-linked for a suitable pair $(\rho,\pi)$ (a property between $\sigma$-centered and $\sigma$-linked, first defined in~\cite{kamo}, see Def.~\ref{def:pirholinked}).
So we use a different (and more cumbersome) relation $\RIII$,
introduced in~\cite{kamo}, where it is also shown that 
$(\rho,\pi)$-linked forcings are $\RIII$-good.

\medskip

\textbf{Regarding the whole diagram:} In~\cite{ten},
starting with the iteration for~\eqref{eq:leftold}, 
a new iteration is constructed to get simultaneously different values for all characteristics:
Assuming  four strongly compact cardinals, the following is consistent (cf.\ Figure~\ref{fig:oldcards}):
\[
\aleph_1 < \addnull < \covnull  <\mathfrak{b}< 
\nonmeager < \covmeager < 
 \mathfrak{d} < \nonnull  < \cofnull < 2^{\aleph_0}.
\]
The essential ingredient 
is the concept of the Boolean ultrapower of a forcing notion. 

In exactly the same way we can expand our new version~\eqref{eq:leftnew} 
to the right hand side, where also the 
characteristics dual to $\mathfrak{b}$ and $\covnull$ are swapped. So we get:
If four strongly compact cardinals are consistent, then so is
the following (cf.\ Figure~\ref{fig:ourcards}):
\[
\aleph_1 < \addnull < \mathfrak{b} <\covnull <\nonmeager<\covmeager<  \nonnull < \mathfrak{d} < \cofnull < 2^{\aleph_0}.
\]

We closely follow the presentation of~\cite{ten}. Several times, we refer to~\cite{ten} and to~\cite{sh} for details in definitions or proofs.
We thank Martin Goldstern and Diego Mej\'{i}a for valuable discussions, and an anonymous referee for a very detailed and helpful report pointing out (and even fixing) several mistakes in the first version of the paper.


\section{Finitely additive measure limits and the \texorpdfstring{$\QII$}{\pdfQII}-forcing.}

\subsection{FAM-limits and random forcing}

We briefly list some basic notation and facts around finite additive measures.
(A bit more details can be found in Section~1 of~\cite{sh}.)

\begin{definition}
\begin{itemize}
\item
A ``partial FAM'' (finitely additive measure) $\Xi'$ is a 
finitely additive probability measure on a sub-Boolean algebra $\mathcal B$ of  $\mathcal P(\omega)$, 
the power set of $\omega$, such that $\{n\}\in\mathcal B$ and $\Xi'(\{n\})=0$ for all $n\in\omega$. We set $\dom(\Xi')=\mathcal B$.
\item 
$\Xi$ is a FAM if it is a partial FAM with $\dom(\Xi)=\mathcal P(\omega)$.
\item  
For every FAM $\Xi$ and bounded sequence of non-negative reals $\bar a=(a_n)_{n\in\omega}$ we can define in the natural way the 
average (or: integral) $\Av_\Xi(\bar a)$, a non-negative real number.
\end{itemize}
\end{definition}

\cite[1.2]{sh} lists several results that informally say:
\begin{equation}\tag{$*$}\label{eq:basic}
\parbox{0.8\columnwidth}{There is a FAM $\Xi$ that assigns the values 
$a_i$ to the sets $A_i$ (for all $i$ in some index set $I$) iff
for each $I'\subseteq I$ finite and $\epsilon>0$
there is an arbitrary large\footnotemark{}
finite $u\subseteq \omega$ such that the counting measure on $u$ for $A_i$ approximates $a_i$ with an error of
at most $\epsilon$, for all $i\in I'$.}
\end{equation}\footnotetext{Equivalently: ``a finite $u$ with arbitrary large minimum'', which is the formulation actually used in most of the results.}

For the size of such an ``$\epsilon$-good approximation'' $u$ to some FAM $\Xi$ we can give an upper bound for $|u|$ 
which only depends on $|I'|$ and $\epsilon$ (and not on $\Xi$):

\begin{lemma}\label{lem:kjwrjio}
Given 
$N,k^*\in\omega$ and $\epsilon>0$, there is an $M\in\omega$  such that:
For all FAMs $\Xi$ and $(A_n)_{n<N}$  there is a nonempty $u\subseteq \omega$
of size ${\le}M$ such that $\min(u)>k^*$ and 
$\Xi(A_n)-\epsilon<\frac{|A_n\cap u|}{|u|}<\Xi(A_n)+\epsilon$
for all $n<N$.
\end{lemma}
\begin{proof} 
We can assume that $\epsilon=\frac{1}{L}$ for an integer $L$.
$\{A_n:\, n\in N\}$ generates the set algebra $\mathfrak B\subseteq \mathcal P(\omega)$. 
Let $\mathcal X$ be the set of atoms of $\mathfrak{B}$. So 
$\mathcal X$ is a partition of $\omega$ of size ${\le}2^N$. 
Set $\mathcal X'=\{x\in \mathcal X:\, \Xi(x)>0\}$.
Every $x\in \mathcal X'$ is infinite, and 
$\sum_{x\in\mathcal X'}\Xi(x)=1$.

Round $\Xi(x)$ to some number 
$\Xi^\epsilon(x)=\ell_x\cdot \frac{1}{L\cdot 2^{N}}$ for some 
integer 
$0\le \ell_x \le L\cdot 2^N$, such that
$|\Xi(x)-\Xi^\epsilon(x)|<\frac{1}{L\cdot 2^{N}}$
and 
$\sum_{x\in \mathcal X'}\Xi^\epsilon(x)$ is still $1$.
So $\sum_{x\in \mathcal X'} \ell_x=L\cdot 2^N$, and we  construct 
$u$ consisting of 
$\ell_x$ many points that are bigger than $k^*$ and in  $x$ (for each $x\in\mathcal X'$).
\end{proof}

We will use the following variants of~\eqref{eq:basic}, regarding the possibility to extend a partial FAM $\Xi'$ to a FAM $\Xi$.
The straightforward, if somewhat tedious, proofs are given in~\cite[1.3(G) and 1.7]{sh}.

\begin{fact}\label{fact:FAMextensions}
Let $\Xi'$ be a partial FAM, and $I$ some index set.
\begin{enumerate}[(a)]
\item\label{item:uf} 
Fix for each $i\in I$ some $A_i\subseteq \omega$.
\\
\emph{If} $A\cap \bigcap_{i\in I'} A_i\neq \emptyset$
for all $I'\subseteq I$ finite and 
 $A\in\dom(\Xi')$ with $\Xi'(A)>0$,
\\
\emph{then} $\Xi'$ can be extended to a FAM $\Xi$ such that $\Xi(A_i)=1$ for all $i\in I$.
\item\label{item:FAMsucc}
Fix for each $i\in I$
some real $b^i$ and some bounded sequence of non-negative reals $\bar a^i=(a^i_k)_{k\in\omega}$.
\\
\emph{If} for each finite partition $(B_m)_{m<m^*}$ of $\omega$
into elements of $\dom(\Xi')$, for each $\epsilon >0$, $k^*\in \omega$, and $I'\subseteq I$ finite 
there is a finite $u\subseteq \omega\setminus k^*$ such that 
\begin{itemize}
\item for all $m<m^*$, $\Xi'(B_m)-\epsilon\le \frac{|B_m\cap u|}{|u|}\le \Xi'(B_m)+\epsilon$, and 
\item for all $i\in I'$, $\frac1{|u|}\sum_{k\in u} a^i_k\ge b^i-\epsilon$,
\end{itemize} 
\emph{then} $\Xi'$ can be extended to a FAM $\Xi$ such that $\Av_{\Xi}(\bar a^i)\ge b^i$ for all $i\in I$.
\end{enumerate}
\end{fact}

We first define what it means for a forcing $Q$ to have FAM limits.

\begin{remark}
Intuitively, this means (in the simplest version): 
Fix a FAM $\Xi$.
We can define for each sequence $q_k$ of conditions that are all
``similar'' (e.g., have the same stem and measure)
a limit $\lim_\Xi \bar q$.
And we find in the $Q$-extension a FAM $\Xi'$ extending $\Xi$, such that
$\lim_\Xi (\bar q)$ forces that the set of $k$ satisfying 
$P(k)\equiv\text{``}q_k\in G\text{''}$ has ``large'' $\Xi'$-measure.
Up to here, we get the notion used in~\cite{MR3513558}
and~\cite{ten} (but there we use ultrafilters instead of FAMs, and ``large''
means being in the ultrafilter).
However, we need a modification:
Instead of single conditions $q_k$ we use a finite 
sequence $(p_\ell)_{\ell\in I_k}$ (where $I_k$ is a fixed, finite interval);
and the condition $P(k)$, which we want to satisfy on a large set, now is
``$\frac{|\{\ell\in I_k:\, p_\ell \in G\}|}{|I_k|}>b$'' for some suitable $b$.
This is the notion used implicitly in~\cite{sh}.
\end{remark}

\begin{notation*}
Let $T^*$ be a compact subtree of $\omega^{{<}\omega}$, for example $T^*=2^{{<}\omega}$. Let $s,t\in T^*$.
Let $S$ be a subtree of $T^*$.
\begin{itemize}
\item $t\rhd s$ means ``$t$ is immediate successor of $s$''.
\item 
$|s|$ is the length of $s$ (i.e.: the height, or level, of $s$).
\item $[t]$ is the set of nodes in  $T^*$ comparable with $t$.
\item 
We set
$\lim(S)=\{x\in \omega^\omega:\, (\forall n\in\omega)\ x\restriction n\in S\}$.
\item $\stem(S)$ is the smallest splitting node of $S$.
With ``$t\in S$ above the stem'' we mean that
$t \in S$ and $t\ge \stem(S)$; or equivalently:
$t \in S$ and $|t|\ge |\stem(S)|$. 
\item
$\Leb$ is the canonical measure on the 
Borel subsets of $\lim(T^*)$.
We also write $\Leb(S)$ instead of $\Leb(\lim(S))$.\footnote{I.e., we define $\Leb([s])$ by induction 
on the height of $s\in T^*$ as follows: 
$\Leb(T^*)=1$, and if $s$ has $n$ many immediate successors in $T^*$, then $\Leb([t])=\frac{\Leb([s])}{n}$ for any 
such successor. This defines a measure on each basic clopen set, which
in turn defines a (probability) measure on the Borel subsets of $\lim(T^*)$
(a closed subset of $\omega^\omega$).}
\end{itemize}
\end{notation*}

We fix, for the rest of the paper, 
an interval partition $\bar I=(I_k)_{k\in\omega}$   of $\omega$
such that $|I_k|$ converges to infinity.
We will use forcing notions $Q$ satisfying the following setup:
\begin{assumption}\label{asm:qlim}
\begin{itemize}
\item $Q'\subseteq Q$ is dense and the domain of 
functions $\stem$ and $\loss$, where 
$\stem(q)\in H(\aleph_0)$ and 
$\loss(q)$ is a non-negative rational.
\item For each $\epsilon>0$
the set $\{q\in Q':\, \loss(q)<\epsilon\}$ is dense (in $Q'$ and thus in $Q$).
\item 
$\{p\in Q':\, (\stem(p),\loss(p))=(\stem^*,\loss^*)\}$
is $\lfloor\frac1{\loss^*}\rfloor$-linked. 
I.e., each $\lfloor\frac1{\loss^*}\rfloor$ many such conditions are compatible.\footnote{In ~\cite[2.9]{sh},
$\stem$ and $\loss$ are called $h_2$ and $h_1$; and
instead of $I_k$ the interval is called  $[n^*_k,n^*_{k+1}-1]$. Moreover,
in \cite{sh} the sequence $(n^*_k)_{k\in \omega}$ is one of the parameters of a
``blueprint'', whereas we assume that the $I_k$ are fixed.}
\end{itemize}
\end{assumption}

In this paper, $Q$ will be one of the following two forcing notions: random forcing, or $\QII$ (as defined in Definition~\ref{def:QII}).
We will now specify the instance of random forcing that we will use:

\begin{definition}\label{def:random}
\begin{itemize}
\item A random condition is a tree $T\subseteq 2^{{<}\omega}$
such that $\Leb(T\cap [t])>0$ for all $t\in T$.
\item $\stem(T)$ is the stem of $T$ (i.e., the shortest splitting node).
\item 
If $\Leb(T)=\Leb([\stem(T)])$, we set $\loss(T)=0$. Otherwise,
let $m$ be the maximal natural number such that 
\[
\Leb(T)> \Leb([\stem(T)])(1-\frac{1}{m})
\]
and set\footnote{In~\cite{sh}, this is implicit in 2.11(f).}
$\loss(T)=\frac1m$.
\end{itemize}
\end{definition}
Note that $\Leb(T)\ge 2^{-|\stem(T)|}(1-\loss(T))$ (and the inequality is strict if $\loss(T)>0$).


Note that this definition of random forcing
satisfies Assumption~\ref{asm:qlim}
(with $Q'=Q$).

\begin{definition}\label{def:famlim} 
Fix $Q$ and functions $(\stem,\loss)$ as in Assumption~\ref{asm:qlim}, 
a FAM $\Xi$
and a function $\lim_\Xi: Q^\omega \to Q$.
Let us call the objects mentioned so far a ``limit setup''.
Let a $(\stem^*,\loss^*)$-sequence be a sequence
$(q_\ell)_{\ell\in\omega}$ of $Q$-conditions such that
$\stem(q_\ell)=\stem^*$ and 
$\loss(q_\ell)=\loss^*$ for all $\ell\in\omega$.

We say ``$\lim_\Xi$ is a strong FAM limit for intervals'', if 
the following is satisfied: Given
\begin{itemize}
\item a pair $(\stem^*,\loss^*)$,
$j^*\in\omega$, and $(\stem^*,\loss^*)$-sequences $\bar q^j$ 
for $j<j^*$,
\item $\epsilon>0$, $k^*\in\omega$,
\item $m^*\in\omega$ and a partition of $\omega$ into 
sets $B_m$ ($m\in m^*$), and 
\item a condition $q$ stronger than all $\lim_\Xi(\bar q^j)$  for all $j<j^*$, 
\end{itemize}
there is a finite $u\subseteq \omega\setminus k^*$ and a $q'$ stronger than $q$ such that
\begin{itemize}
\item $\Xi(B_m)-\epsilon<\frac{|u\cap B_m|}{|u|}<\Xi(B_m)+\epsilon$ for $m<m^*$,
\item  
$
\frac1{|u|}\sum_{k\in u}\frac{|\{\ell\in I_k:\, q'\le q^j_\ell\}|}{|I_k|}\ \ge\  1-\loss^*-\epsilon
$ for $j<j^*$
\end{itemize}
\end{definition}
(We are only interested in $\lim_\Xi(\bar q)$
for $\bar q$ as above, so we can set $\lim_\Xi(\bar q)$
to be undefined or some arbitrary value for other $\bar q\in Q^\omega$.)

The motivation for this definition is
the following:

\begin{lemma}\label{lem:extension}
Assume that $\lim_\Xi$ is such a limit. 
Then there is a $Q$-name $\Xi^+$ such that 
for every $(\stem^*,\loss^*)$-sequence $\bar q$ 
the limit $\lim_\Xi(\bar q)$ forces 
$\Xi^+(A_{\bar q})\ge 1-\sqrt{\loss^*}$, where
\begin{equation}\label{eq:ajkrhw}
A_{\bar q}=\{k\in \omega:\,  |\{\ell\in I_k:\, q_\ell\in G\}|\ge |I_k|\cdot  (1-\sqrt{\loss^*})\}    
\end{equation}
\end{lemma}

\begin{proof}
Work in the $Q$-extension. Now $\Xi$ is a partial FAM.
Let $J$ enumerate all suitable sequences $\bar q\in V$ with $\lim_{\Xi}(\bar q)\in G$, and for such a sequence $\bar q^j$
set $a^j_k=\frac{|\{\ell\in I_k:\, q^j_\ell\in G\}|}{|I_k|}$, 
and $b^j=1-\loss^*$.
Using that $\Xi$ satisfies Definition~\ref{def:famlim}, we can  apply 
Fact~\ref{fact:FAMextensions}(\ref{item:FAMsucc}), we can extend 
$\Xi$ to some FAM $\Xi^+$ such that 
$
\Av_{\Xi^+}(\bar a^j)\ge 1-\loss^*
$ for $j<j^*$.
So
$\Xi^+(A_{\bar q^j})+ (1-\Xi^+(A_{\bar q^j}))\cdot (1-\sqrt{\loss^*})\ge
\Av_{\Xi^+}(a^j_k)\ge 1-\loss^*$,
 and thus $\Xi^+(A_{\bar q^j})\ge 1-\sqrt{\loss^*}$.
\end{proof}

\begin{definition}
$(Q,\stem,\loss)$ as in Assumption~\ref{asm:qlim}
``has strong FAM limits for intervals'', if for every FAM $\Xi$ 
there is a function
$\lim_\Xi$ that is a strong FAM limit for intervals.
\end{definition}

\begin{lemma}\label{lem:randomhaslimits} \cite{sh}
Random forcing has strong FAM-limits for intervals.
\end{lemma}

\begin{proof}
$\lim_\Xi$ is implicitly defined in \cite[2.18]{sh}, in the following way:
Given a sequence $r_\ell$
with $(\stem(p_\ell),\loss(p_\ell))=(\stem^*,\loss^*)$,
we can set $r^*=[\stem^*]$ and $b=1-\loss^*$; and we set
$n^*_k$ such that $I_k=[n^*_k,n^*_{k+1}-1]$.
We now use these objects to apply
\cite[2.18]{sh} (note that (c)($*$) is satisfied). This gives 
$r^\otimes$, and we define $\lim_\Xi(\bar r)$ to be $r^\otimes$.

In \cite[2.17]{sh}, it is shown that this $r^\otimes$ satisfies
Definition~\ref{def:famlim}, i.e., is a limit:
If $r$ is stronger than all limits $r^{\otimes i}$, then $r$ satisfies 
\cite[2.17($*$)]{sh}.
\end{proof}

\subsection{The forcing \texorpdfstring{$\QII$}{\pdfQII}}

We now define $\QII$, a variant of the forcing notion $Q^2$ defined in~\cite{1067}:

\begin{definition}\label{def:QII}
By induction on the height $h\ge 0$, we define a compact 
homogeneous tree $T^*\subset \omega^{<\omega}$, and set 
\begin{equation}\label{eq:blaww3}
\rho(h)\defeq \max(|T^*\cap \omega^h|,h+2)
\quad\text{and}\quad
\pi(h)\defeq ((h+1)^2 \rho(h)^{h+1})^{\rho(h)^h},
\end{equation}
we set $\Omega_s$ to be the set $\{t\rhd s:\, t\in T^*\}$, i.e., the set of immediate successors of $s$,
and define for each $s$ a norm $\mu_s$ on the subsets of $\Omega_s$. In more detail:
\begin{itemize}
\item The unique element of $T^*$ of height $0$ is $\langle\rangle$, i.e.,
$T^*\cap \omega^0=\{\langle\rangle\}$. 

%
\item
We set
\[ 
a(h)=\pi(h)^{h+2},\quad M(h)=a(h)^2, \quad\text{and}\quad \mu_h(n)=\log_{a(h)}\left(\frac{M(h)}{M(h)-n}\right)
\]
for natural numbers $0\le n< M(h)$, and we set
$\mu_h(M(h))=\infty$.

\item For any $s\in T^*\cap \omega^h$, we set $\Omega_s=\{s^\frown \ell:\, \ell\in M(h)\}$
(which defines $T^*\cap \omega^{h+1}$). 
For $A\subset \Omega_s$, we set $\mu_s(A)\defeq \mu_h(|A|)$.
So $|\Omega_s|=M(h)$,  $\mu_s(\emptyset)=0$ and $\mu_s(\Omega_s)=\infty$.
Note that $|A|=|\Omega_s|\cdot\left(1-a(h)^{-\mu_s(A)}\right)$.
\end{itemize}
\end{definition}

We can now define $\QII$:
\begin{definition}
\begin{itemize}


\item
For a subtree $p\subseteq T^*$,
the stem of $p$ is the smallest splitting node. For $s\in p$, we set $\mu_s(p)=\mu_s(\{t\in p:\, t \rhd s\})$.

$\QII$ consists of subtrees $p$ with some stem $s^*$ of height $h^*$
such that $\mu_t(p)\ge 1+\frac1{h^*}$
for all $t\in p$ above the stem.
(So the only condition with $h^*=0$ is the full condition, where all norms are $\infty$.)

$\QII$ is ordered by inclusion.
\item $\stem(p)$ is the stem of $p$.

$\loss(p)$ is defined if there is an $m\ge 2$ satisfying the following,
and in that case $\loss(p)=\frac1m$ for the maximal such $m$:
\begin{itemize}
\item $p$ has stem $s^*$ of height $h^*>3m$,
\item $\mu_s(p)\ge 1+\frac1m$ for all $s\in p$ of height ${\ge}h^*$.
\end{itemize}
We set $Q'=\dom(\loss)$.
\end{itemize}
\end{definition}
By simply extending the stem, we can find for any $p\in \QII$
and $\epsilon>0$ some $q\le p$ in $Q'$ with $\loss(q)<\epsilon$;
i.e., one of the assumptions in~\ref{asm:qlim} is satisfied. 
(The other one is dealt with in Lemma~\ref{lem:QIIbasic}(\ref{item:linked}).)
In particular $Q'\subseteq\QII$ is dense.

We list a few trivial properties of the loss function:
\begin{facts}\label{fact:loss}
Assume $p\in Q'$ with $s=\stem(p)$ of height $h$.
\begin{enumerate}[(a)]
\item\label{item:losstrivialities}
$\loss(p)<1$, $\mu_s(p)\ge 1+\loss(p)$  for any $s$ above the stem,
and $\loss(p)>\frac{3}{h}$. 
\item\label{item:lossbla}
If $q$ is a subtree of $p$ such that all
norms above the stem are $\ge 1+\loss(p)-\frac{2}h$,
then $q$ is a valid $\QII$-condition.
\item\label{item:productlimit}
    $\prod_{\ell =h}^{\infty} (1-\frac1{\ell^2})=1-\frac1{h}>1-\frac{\loss(p)}{3}$.
\end{enumerate}
\end{facts}

\begin{lemma}\label{lem:ramsey}
Let $s\in T^*$ be of height $h$ and $A\subset\Omega_s$.
\begin{enumerate}[(a)]
\item\label{item:count}
If $\mu_s(A)\ge 1$, then
$|A|\ge |\Omega_s|\cdot (1-\frac1{h^2})$.
\item\label{item:minusone} If $A\subsetneq \Omega_s$, i.e., $A$ is a proper subset, then $\mu_s(A\setminus \{t\})> \mu_s(A)-\frac1h$ for $t\in A$.
\item\label{item:intersecth} For $i<\pi(h)$, assume that
$A_i\subseteq \Omega_s$ satisfies $\mu_s(A_i)\ge x$. 
Then $\mu_s(\bigcap_{i\in \pi(h)} A_i)> x-\frac1h$.

\item\label{item:ramsey}
For $i<I$ (an arbitrary finite index set)
pick proper subsets $A_i\subsetneq \Omega_s$ 
such that $\mu_s(A_i)\ge x$,
and assign weighs $a_i$ to $A_i$ such that $\sum_{i\in I}a_i=1$. Then
%
%
\begin{equation}\label{eq:fubini}
\mu_s(B)> x-\frac1h \quad\text{for}\quad B\defeq \Big\{t\in \Omega_s:\, \sum_{t\in A_i}a_i > 
1-\frac1{h^2}
\Big\}.
\end{equation}

\end{enumerate}
\end{lemma}

\begin{proof}
\begin{enumerate}[(a)]
\item
Trivial, as $a(h)^{-\mu_s(A)}\le \frac1{a(h)}<\frac1{h^2}$.
\item
$
\mu_s(A\setminus \{t\})
=\log_{a(h)}(|\Omega_s|)-\log_{a(h)}(|\Omega_s|-|A|+1)
\ge
$\\\phantom{.}\hfill$\ge
\log_{a(h)}(|\Omega_s|)-\log_{a(h)}(2(|\Omega_s|-|A|))
\ge  \mu_s(A)-\log_{a(h)}(2) >  \mu_s(A)-\frac1h$.
\item
$\mu_s(\bigcap_{i\in \pi(h)} A_i)=
\log_{a(h)}(|\Omega_s|)-\log_{a(h)}(|\Omega_s|-|\bigcap_{i\in \pi(h)} A_i|)
=
$\\\phantom{.}\hfill$=
\log_{a(h)}(|\Omega_s|)-\log_{a(h)}(|\bigcup_{i\in \pi(h)} (\Omega_s-A_i)|)
\ge 
$\hfill\phantom{.}\\\phantom{.}\hfill$\ge
\log_{a(h)}(|\Omega_s|)-\log_{a(h)}\big(\pi(h)\cdot \max_{i\in \pi(h)} |\Omega_s-A_i|\big)
\ge x -\log_{a(h)}(\pi(h))
> x - \frac1h$.
\item
Set $y=\sum_{i\in I} a_i\cdot |A_i|$. On the one hand, $y\ge |\Omega_s|\cdot (1-a(h)^{-x})$.
On the other hand, $y= \sum_{t\in \Omega_s}\sum_{t\in A_i}a_i
\le |B|+ (|\Omega_s\setminus B|)\cdot(1-\frac1{h^2})$.

So $|B|\ge |\Omega_s|(1-h^2 a(h)^{-x})> |\Omega_s|(1- a(h)^{-(x-\frac1h)})$,
as $a(h)^{\frac1h}> \pi(h)> h^2$.\qedhere
\end{enumerate}
\end{proof}

$\QII$ is not $\sigma$-centered, but it satisfied a property,
first defined in~\cite{kamo}, which
is between $\sigma$-centered and $\sigma$-linked:
\begin{definition}\label{def:pirholinked}
Fix $f,g$ functions from $\omega$ to $\omega$ converging to infinity.
$Q$ is $(f,g)$-linked
if there are $g(i)$-linked $Q^i_j\subseteq Q$ for $i<\omega,j<f(i)$
such that each $q\in Q$ is in every $\bigcup_{j<f(i)}Q^i_j$ for sufficiently large $i$.
\end{definition}

Recall that we have defined
$\rho$ and $\pi$ in~\eqref{eq:blaww3}.

\begin{lemma}\label{lem:QIIbasic}
\begin{enumerate}[(a)]
\item\label{item:linked}
If $\pi(h)$ many conditions $(p_i)_{i\in \pi(h)}$ have a common node $s$ above their stems, $|s|=h$,
then there is a $q$ stronger than each $p_i$.
\item\label{item:pr} $\QII$ is $(\rho,\pi)$-linked 
(In particular it is ccc).
\item\label{item:ed} The $\QII$-generic real $\eta$ is eventually different (from every real in $\lim(T^*)$, and therefore from every real in $\omega^\omega$ as well).
\item\label{item:ghgh}
$\Leb(p)\ge \Leb\big([\stem(p)]\big)\cdot \big(1-\frac12\loss(p)\big)$; more explicitly:
for any $h>|\stem(p)|$,
\[
\frac{|p\cap \omega^h|}{|T^*\cap \omega^h\cap [\stem(p)]|}\ge 1-\frac12\loss(p).
\]
\item\label{item:subrand}  
$Q'$ (which is a dense subset of $\QII$) is an incompatibility-preserving subforcing of random forcing,
where we use the variant\footnote{We can use Definition~\ref{def:random}, replacing $2^\omega$ with
$\lim(T^*)$.} of random forcing on $\lim(T^*)$ instead of $2^\omega$.
Let $B'$ be the the sub-Boolean-algebra 
of $\mathrm{Borel}/\mathrm{Null}$ generated by 
$\{\lim(q):\, q\in Q'\}$.
Then $Q'$ is dense in $B'$.
\end{enumerate}
\end{lemma}
(Here, 
$\mathrm{Borel}$ refers to the set of Borel subsets of $\lim(T^*)$.
In the following proof, we will denote the equivalence class
of a Borel set $A$ by $[A]_\mathcal N$.)

\begin{proof}
\begin{enumerate}[(a)]
\item
Set $S=[s]\cap \bigcap_{i<\pi(h)} p_i$.
According to~\ref{lem:ramsey}(\ref{item:intersecth}),
for each $t\in S$ of height $h'\ge h$,
the successor set has norm bigger than $1+\nicefrac1h-\nicefrac1{h'}> 1$, 
so in particular there is a branch $x\in S$,
and $S\cap [x\restriction 2h]$ is a valid condition
stronger than all $p_i$.
\item
For each $h\in\omega$, enumerate
$T^*\cap \omega^h$ as $\{s^h_1,\dots,s^h_{\rho(h)}\}$,
and set $Q^h_i=\{p\in \QII:\, s^h_i\in p\text{ and }|\stem(p)|\le h\}$.
So for all $h$,
$Q^h_i$ is $\pi(h)$-linked, and $p\in \bigcup_{i<\rho(h)}Q^h_i$ for
all $p\in Q$ with $|\stem(p)|\le h$.
\item
Use \ref{lem:ramsey}(\ref{item:minusone}).

\item 
Use \ref{lem:ramsey}(\ref{item:count}) and the definition of $\loss$.
\item
As in the previous item, we get that $\Leb(p\cap [t])>0$ whenever $p\in Q'$
and $t\in p$. So $Q'$ is a subset of random forcing. As both sets are ordered by inclusion,
$Q'$ is a subforcing.
If $q_1,q_2\in Q'$ and 
$q_1,q_2$ are compatible as a random conditions, then $q_1\cap q_2$ has arbitrary high nodes, in particular a node above both stems,
which implies that $q_1$ is compatible with $q_2$ in $\QII$ and therefore in $Q'$.
It remains to show that $Q'$ is dense in $B'$. 
It is enough to show:
If $x\ne 0$ in $B'$ has the form
$x=\bigwedge_{i<i^*}[\lim(q_i)]_\mathcal N\wedge \bigwedge_{j<j^*}[\lim(T^*)\setminus \lim(q_j)]_\mathcal N$
then there is some $q\in Q'$ with $[\lim(q)]_\mathcal N<x$.
Note that $0\ne x=[A]_\mathcal N$
for $A= \lim\left(\bigcap_{i<i^*}q_i\right)\setminus \bigcup_{j<j^*}\lim(q_j)$, 
so pick some $r\in A$ and pick $h>i^*$ large enough such that $s=r\restriction h$ is 
not in any $q_j$.
Then any $q\in Q'$ stronger than all $ q_i\cap [s]$ (for $i<i^*$) is as required.
\qedhere
\end{enumerate}
\end{proof}

\begin{lemma}\label{lem:QIIhaslimits}
$\QII$ has strong FAM-limits for intervals.
\end{lemma}

\begin{proof}
Let $(p_\ell)_{\ell\in\omega}$ 
be a $(s^*,\loss^*)$-sequence, 
$s^*$ of height $h^*$. 
Set
$\tilde\zeta^{h^*}=0$ and 
\[\textstyle\tilde\zeta^h\defeq 1-\prod_{m=h^*}^{h-1} (1-\frac{1}{m^2})\text{ for }h>h^*.\]
This is a strictly increasing sequence below $\frac13\loss^*$, cf.\ Fact~\ref{fact:loss}(\ref{item:productlimit}).
Also,
all norms in all conditions of the sequence are at least $1+\loss^*$,
cf.\ Fact~\ref{fact:loss}(\ref{item:losstrivialities}).

We will first construct $(q_k)_{k\in\omega}$  with stem $s^*$
and all norms  $>1+\loss^*-\frac{1}{h^*}$
such that $q_k$ forces
$\frac{|\{\ell\in I_k:\, p_\ell\in G\}|}{|I_k|}> 1-\frac13\loss^*$.
We will then use $\bar q$ to define $\lim_\Xi(\bar p)$, 
and in the third step show that it is as required.

\textbf{Step 1:}
So let us define $q_k$. Fix $k\in\omega$.
\begin{itemize}
\item  Set $X_{t}=\{\ell\in I_k:\, t\in p_\ell\}$ and 
$Y_{h}=\left\{ t\in [s^*]\cap \omega^h:\, |X_{t}|\ge |I_k|\cdot (1- \tilde\zeta^h) \right\}$.
\item We define $q_k$ by induction on the level, such that
$q_k\cap \omega^h\subseteq Y_h$. The stem is $s^*$. (Note that 
 $X_{s^*}=I_k$ and so $s^*\in Y_{h^*}$.) 
For $s\in q_k\cap \omega^h$ (and thus, by induction hypothesis, in $Y_h$),
we set $q_k\cap [s]\cap \omega^{h+1}=[s]\cap Y_{h+1}$, i.e., a successor $t$ of $s$ is in $q_k$ iff it is $Y_{h+1}$. Then $\mu_s(q_k)> 1+\loss^*-\frac{1}{h}$.

Proof: Set $I=X_s$. By induction, $|X_s|\ge |I_k|\cdot (1-\tilde\zeta^h)$.
For $\ell\in I$, set $A_\ell= p_\ell\cap [s]\cap \omega^{h+1}$, i.e., 
the immediate successors of $s$ in $p_\ell$. 
Obviously $\mu_s(A_\ell)\ge 1+\loss^*$.
We give each $A_\ell$ equal weight $a_\ell=\frac1{|I|}$.
According to~\eqref{eq:fubini}, the set $B=\{t\rhd s:\, |\{\ell\in X_s:\, t\in A_\ell\}|\ge |I|\cdot (1-\frac1{h^2})\}$ 
has norm $>1+\loss^*-\frac1h$.
\item $q_k$ 
forces that $p_\ell\in G$ for $\ge |I_k|\cdot (1-\frac12\loss^*)$ many $\ell\in I_k$.

Proof: Let $r<q_k$ have stem $s'$ of length $h'$, without loss of generality $h'>|I_k|+1$. 
As $s'\in Y_{h'}$, there are $>|I_k|\cdot (1-\frac13\loss^*)$ many 
$\ell\in I_k$ such that $s'\in p_\ell$. 
So we can find a a condition $r'$ stronger than $r$ and all these
$p_\ell$ (as these are at most $|I_k|+1\le h'$ many conditions all
containing $s'$ above the stem).
\end{itemize}

\textbf{Step 2:}
Now we use $(q_k)_{k\in\omega}$ to construct by induction on the height
$q^*=\lim_\Xi(\bar p)$,
a condition with stem $s^*$ and all norms ${\ge}1+\loss^*-\frac2h$ 
such that for all $s\in q^*$ of height $h\ge h^*$,
\begin{equation}\tag{$*$}\label{eq:Z}
\Xi(Z_s)\ge 1-\tilde\zeta^h
%
,\ \text{for}\quad
Z_s\defeq\{k\in\omega:\, s\in q_k\}.
\quad\text{So }
\Xi(Z_s)>1-\frac13\loss^*.
\end{equation}
Note that $Z_{s^*}=\omega$, so~\eqref{eq:Z} is satisfied for $s^*$.
Fix an $s\ge s^*$ satisfying~\eqref{eq:Z}. Set
$A(k)$ to be the $s$-successors in $q_k$
for each $k\in Z_s$.
Enumerate the (finitely many) $A(k)$
as $(A_i)_{i\in I}$.
Clearly $\mu_s(A_i)> 1+\loss^*-\frac1h$. Assign to $A_i$ the weight
$a_i=\frac1{\Xi(Z_s)}\Xi(\{k\in Z_s:\, A(k)=A_i\})$.
Again using~\eqref{eq:fubini}, 
$\mu_s(B)\ge 1+\loss^*-\frac2h$, where $B$ consists of those successors $t$ of $s$
such that 
\[
1-\frac1{h^2}<\sum_{t\in A_i}a_i = \frac1{\Xi(Z_s)}\Xi(\{k\in Z_s:\, t\in q_k\})\le \frac1{\Xi(Z_s)}\Xi(Z_t)
.\] 
So every $t\in B$ satisfies $\Xi(Z_t)>\Xi(Z_s)(1-\frac{1}{h^2})\ge \tilde\zeta^{h+1}$, i.e., satisfies~\eqref{eq:Z}.
So we can use $B$ as the set of $s$-successors in $q^*$. 

This defines $q^*$, which is a valid condition by Fact~\ref{fact:loss}(\ref{item:lossbla}).

\textbf{Step 3:}
We now show that this limit works: As in Definition~\ref{def:famlim}, fix $m^*$, $(B_m)_{m<m^*}$, $\epsilon$, $k^*$, $i^*$ and
sequences $(p^i_\ell)_{\ell<\omega}$ for $i<i^*$,
such that $(\stem(p^i_\ell),\loss(p^i_\ell))=(\stem^*,\loss^*)$.

For each $i<i^*$,
$\bar q^i=(q^i_k)_{k\in\omega}$ is defined from $\bar p^i=(p^i_\ell)_{\ell\in\omega}$,
and in turn defines the limit $\lim_\Xi(\bar p^i)$.
Let $q$ be stronger than all $\lim_\Xi(\bar p^i)$.

Let $M$ be as in Lemma~\ref{lem:kjwrjio}, for $N=m^*+i^*$.
So for any $N$ many sets there is a $u$ of size at most $M$ (above $k^*$)
which approximates the measure well.
We use the following $N$ many  sets:
\begin{itemize}
\item $B_m$ (for $m<m^*$).
\item Fix an $s\in q$ of height $h>M\cdot i^*$; and use
the $i^*$ many sets
$Z^i_s\subseteq \omega$ defined in~\eqref{eq:Z}.
\end{itemize}
Accordingly, there is a $u$ (starting above $k^*$) of size ${\le}M$
with
\begin{itemize}
\item
$\Xi(B_m)-\epsilon\le \frac{|B_m\cap u|}{|u|}\le \Xi(B_m)+\epsilon$
for each $m<m^*$, and 
\item $\frac{|Z^i_s\cap u|}{|u|}\ge 1-\frac13\loss^*-\epsilon$ for each $i<i^*$.
\end{itemize} 
So for each $i\in i^*$ there
are at least $|u|\cdot (1-\frac12\loss^*-\epsilon)$ many 
$k\in u$ with $s\in q^i_k$. 
There is a condition $r$ stronger than $q$
and all those $q^i_k$
(as ${\le}M\cdot i^*+1$ many conditions of height $h>M\cdot i^*$
with common node $s$ above their stems are compatible).
So $r$
forces, for all $i<i^*$ and $k\in u\cap Z^i_s$, that
$q^i_k\in G$ and therefore that 
$|\{\ell\in I_k:\, p^i_\ell\in G\}|\ge |I_k|(1-\frac13\loss^*)$.
By increasing $r$ to some $q'$, we can assume that $r$ decides which $p^i_\ell$ are in $G$
and that $r$ is actually stronger than each $p^i_\ell$ decided to be in $G$.
So all in all we get $q'\le q$ such that 
\[
\frac1{|u|}\sum_{k\in u}\frac{|\{\ell\in I_k:\, q'\le p^j_\ell\}|}{|I_k|}
\ \ge\  
\frac1{|u|}|\{k\in u:\, k\in Z^j_s\}|(1-\frac13\loss^*)
\ >\  
1-\loss^*-\epsilon,
\]
as required.
\end{proof}
\pagebreak[1]

%

\section{The left hand side of Cicho\'n's diagram}\label{sec:partA}

We write $\mathfrak{x}_1$ for 
$\addnull$, $\mathfrak{x}_2$ for 
$\mathfrak{b}$ (which will also be $\addmeager$),
$\mathfrak{x}_3$ for $\covnull$ and
$\mathfrak{x}_4$ for $\nonmeager$.

\subsection{Good iterations and the \texorpdfstring{$\mylin$}{LCU} property}

We want to show that some forcing $\Pa$
results in $\mathfrak{x}_i=\lambda_i$ (for $i=1\dots4$).
So we have to show two ``directions'', 
$\mathfrak{x}_i\le \lambda_i$ 
and $\mathfrak{x}_i\ge \lambda_i$. 

For $i=1,3,4$ (i.e., for all the 
characteristics on the left hand side apart from
$\mathfrak{b}=\addmeager$),
the direction $\mathfrak{x}_i\le \lambda_i$  will be given
by the fact that $\Pa$ is $(R_i,\lambda_i)$-good
for a suitable relation $R_i$. (For $i=2$, i.e., the unbounding number, we will have to work more.)

We will use the following relations:
\begin{definition}
\begin{itemize}
\item[1.] 
Let $\mathcal C$  be the set
of strictly positive rational sequences $(q_n)_{n\in\omega}$ such that
$\sum_{n\in\omega}q_n\le 1$.\footnote{It is easy to see that $\mathcal C$ is homeomorphic to $\omega^\omega$, when
we equip the rationals with the discrete topology 
and use the product topology.}
Let $\RI\subseteq \mathcal C^2$ be defined by: $f \RI g$ if
$(\forall^*n\in\omega)\, f(n)\le g(n)$.
\item[2.] $\RII\subseteq  (\omega^\omega)^2$ is defined by: $f \RII g$  if $(\forall^* n\in\omega)\, f(n)\le g(n)$.
\item[4.] $\RIV\subseteq (\omega^\omega)^2$ is defined by: $f \RIV g$  if $(\forall^* n\in\omega)\, f(n)\neq g(n)$.
\end{itemize}
\end{definition}

So far, these relations fit the usual framework of goodness, as
introduced in~\cite{MR1071305} and~\cite{MR1129144} and summarized, e.g., in~\cite[6.4]{BJ} or~\cite[Sec.\ 3]{MR3513558} or~\cite[Sec.\ 2]{MR3047455}.
For $\mathfrak{x}_3$, i.e., $\covnull$, we will use a relation $\RIII$
that does not fit this framework (as the range of the relation is not a Polish space).
Nevertheless, the property ``$(\RIII,\lambda)$-good'' behaves just as in the usual framework
(e.g., finite support limits of good forcings are good, etc.).
The relation $\RIII$ was implicitly used by Kamo and Osuga~\cite{kamo},
who investigated $(\RIII,\lambda)$-goodness.\footnote{They use the notation $(*_{c,h}^{<\lambda})$, cf.~\cite[Def.\ 6]{kamo}.} It was also used in~\cite{gaps}; a unifying notation for goodness (which works for  the
usual cases as well as relations such as $\RIII$) is given in~\cite[§4]{miguel}.

\begin{definition}\label{def:mathcalE}
We call a set $\mathcal{E}\subset\omega^\omega$ an $\RIII$-parameter, if 
for all $e\in \mathcal E$
\begin{itemize}
\item $\lim e(n)=\infty$, $e(n)\le n$, $\lim(n-e(n))=\infty$,
\item there is some $e'\in\mathcal E$ such that $(\forall^*n)\, e(n)+1\le e'(n)$, and
\item for all countable $\mathcal E'\subseteq \mathcal E$ there is some 
$e\in\mathcal E$ such that for all $e'\in \mathcal E'$
$(\forall^*n)\, e(n)\ge e'(n)$.
\end{itemize}
\end{definition}
Note that such an $\RIII$-parameter of size $\aleph_1$ exists.
This is trivial if we assume CH (which we could in this paper), but also true without this assumption, see~\cite[4.20]{miguel}.
Recall that $\rho$ and $\pi$ were defined in~\eqref{eq:blaww3}.
\begin{definition}\label{def:b}
We fix, for the rest of the paper, an $\RIII$-parameter $\mathcal E$ of size $\aleph_1$, and set
\begin{multline*}
b(h)=(h+1)^2\rho(h)^{h+1},
\quad
\mathcal S=\big\{ \psi\in\mathcal \prod_{h\in\omega}P(b(h)):\, (\forall h\in\omega)\,|\psi(h)|\le \rho(h)^{h}  \big\},
\\
{\mathcal S}_e=\big\{ \phi\in\mathcal \prod_{h\in\omega}P(b(h)):\, (\forall h\in\omega)\,|\phi(h)|\le \rho(h)^{e(h)}  \big\}
\quad\text{and}\quad
\hat{\mathcal S}=\bigcup_{e\in\mathcal E}\mathcal S_e.
\end{multline*}
We can now define the relation for $\covnull$:
\begin{itemize}
\item[3.] $\RIII\subseteq \mathcal S\times \hat{\mathcal S}$ is defined by: 
$\psi \RIII \phi$ iff $(\forall^*n\in\omega)\, \phi(n)\not\subseteq \psi(n)$. 
\end{itemize}\end{definition}

Note that $\mathcal S_e\subset \hat{\mathcal S}\subset \mathcal S$ and that
$\mathcal S_e$ and $\mathcal S$ are Polish spaces. 
Assume that $M$ is a forcing extension of $V$ by either a ccc forcing
(or by a $\sigma$-closed forcing).
Then $\mathcal E$ is
an ``$\RIII$-parameter'' in $M$ as well, and
we can evaluate in $M$ for each $e\in \mathcal E$ the sets 
$\mathcal S^M_e$ and $\mathcal S^M$, as well as 
$\hat{\mathcal S}^M=\bigcup_{e\in\mathcal E}\mathcal S_e^M$.
Absoluteness gives $\mathcal S^V_e=\mathcal S^M_e\cap V$
and $\hat{\mathcal S}^V=\hat{\mathcal S}^M\cap V$.

\begin{definition}
Fix one of these relations $R\subseteq X\times Y$.
\begin{itemize}
\item
We say ``$f$ is bounded by $g$'' if $f\Rel g$; and, for $\mathcal Y\subseteq \omega^\omega$, 
``$f$ is bounded by $\mathcal Y$'' if
$(\exists y\in \mathcal Y)\, f \Rel y$. We say ``unbounded'' for ``not bounded''. (I.e., $f$ is unbounded by $\mathcal Y$ if $(\forall y\in \mathcal Y)\,\lnot f\Rel y$.) 
\item
We call $\mathcal X$  an $\Rel$-unbounded family, if $\lnot (\exists g)\,(\forall x\in \mathcal X)x\Rel g$, and an $\Rel$-dominating family if $(\forall f)\,(\exists x\in \mathcal X)\, f\Rel x$. 

   \item
Let $\mathfrak{b}_i$ be the minimal size of an $\Rel_i$-unbounded family,
   \item and let  $\mathfrak{d}_i$ be the minimal size
      of an $\Rel_i$-dominating family.
\end{itemize}
\end{definition}

We only need the following connection between $\Ri$ and the cardinal characteristics: 
\begin{lemma}\label{lem:connection}
\begin{itemize}
\item[1.]
$\addnull=\mathfrak{b}_1$ and $\cofnull=\mathfrak{d}_1$.
\item[2.]
$\mathfrak{b}=\mathfrak{b}_2$ and $\mathfrak{d}=\mathfrak{d}_2$.
\item[3.]
$\covnull\le\mathfrak{b}_3$ and $\nonnull\ge\mathfrak{d}_3$.
\item[4.]
$\nonmeager=\mathfrak{b}_4$ and $\covmeager=\mathfrak{d}_4$.
\end{itemize}
\end{lemma}
\begin{proof}
(2) holds by definition. (1) can be found in~\cite[6.5.B]{BJ}.
(4) is a result of~\cite{MR671224,MR917147},  
cf.~\cite[2.4.1~and~2.4.7]{BJ}.

To see (3), we work in the space $\Omega=\prod_{h\in\omega} b(h)$,
with the $b$ defined in Definition~\ref{def:b} and 
the usual (uniform) measure.
It is well known that we get the same
values for the characteristics $\covnull$ 
and $\nonnull$ 
whether we define them using $\Omega$ or, as usual, $2^\omega$ (or $[0,1]$ for that matter, etc). 
Given $\psi\in\mathcal S$, note that 
\[ N_\psi=\{\eta\in\Omega:\,  (\exists^\infty h)\, \eta(h)\in \psi(h)\}\]
is a Null set, as 
$\{\eta\in\Omega:\, (\forall h>k)\, \eta(h)\notin\psi(h)\}$
has measure $\prod_{h>k}(1-\frac{|\psi(h)|}{b(h)})\ge \prod_{h>k}(1-\frac1{(h+1)^3})$, which converges to $1$ for $k\to\infty$.

Let $\mathcal A\subseteq \mathcal S$ be an $\RIII$-unbounded family.
So for every $\phi\in\hat{\mathcal S}$ there is some $\psi\in A$
such that $(\exists^\infty h)\, \psi(h)\supseteq \phi(h)$.
In particular, for each $\eta\in\Omega$, there is 
a $\psi\in A$ with $\eta\in N_\psi$; i.e., $\covnull\le |\mathcal A|$. 

Analogously, 
let $X$ be a non-null set (in $\Omega$). For each $\psi$
there is an $x\in X\setminus N_\psi$, 
so $\phi_x(n)=\{x(n)\}$ satisfies $\psi \RIII \phi_x$.
\end{proof}

\begin{remark}
As shown implicitly in~\cite{kamo}, and explicitly in~\cite[4.22]{miguel}, we actually 
get $\covnull\le c^\exists_{b,\rho^{\mathrm{Id}}}\le \mathfrak b_3$.
\end{remark}

\begin{definition}
Let $P$ be a ccc forcing, $\lambda$ an uncountable regular cardinal, and $\Rel_i\subseteq X\times Y$ one of the relations above
(so for $i=1,2,4$, $Y=X$, and for $i=3$ $Y=\hat{\mathcal S}_e$).
$P$ is $(\Rel_i,\lambda)$-good, if
for each $P$-name $r$ for an element of $Y$ there is (in $V$) a nonempty
set $\mathcal Y\subseteq Y$ of size ${<}\lambda$
such that every $f\in X$ (in $V$) that is $\Rel_i$-unbounded by $\mathcal Y$ is forced to be $\Rel_i$-unbounded by $r$ as well.
\end{definition}

Note that $\lambda$-good trivially implies $\mu$-good if $\mu\ge\lambda$
are regular.

\begin{lemma}\label{lem:gettinggood}
Let $\lambda$ be uncountable regular.
\begin{itemize}
\item[a.] Forcings of size ${<}\lambda$ are $(\Rel_i,\lambda)$-good. In particular, Cohen forcing is 
$(\Rel_i,\aleph_1)$-good.
\item[b.] A FS ccc iteration of $(\Rel_i,\lambda)$-good forcings (and in particular, a composition of two such forcings) is $(\Rel_i,\lambda)$-good.
\item[1.] A sub-Boolean-algebra of the random algebra is $(\RI,\aleph_1)$-good. Any 
$\sigma$-centered forcing notion is $(\RI,\aleph_1)$-good.
\item[3.] A $(\rho,\pi)$-linked forcing is $(\RIII,\aleph_1)$-good
(for the $\rho,\pi$ of Definition~\ref{def:QII}).
\end{itemize}
\end{lemma}

\begin{proof}
\textbf{(a\&b):}
For $i=1,2,4$ this is proven in~\cite{MR1071305}, cf.~\cite[6.4]{BJ}. 
The same proof works for $i=3$, as shown in~\cite[Lem.~12,13]{kamo}.
The proof for the uniform framework can be found in~\cite[4.10,4.14]{miguel}.

\textbf{(1)}
follows from~\cite{MR1071305} and~\cite{MR1022984}, cf.~\cite[6.5.17--18]{BJ}.

\textbf{(3)} is shown in~\cite[Lem.~10]{kamo}, cf.~\cite[Lem.~4.24]{miguel}; as our choice
of $\pi$, $\rho$ and $b$ (see Definition~\ref{def:b}) satisfies
$\pi(h)\ge b(h)^{\rho(h)^h}=((h+1)^2\rho(h)^{h+1})^{\rho(h)^h}$.
\end{proof}


Each relation $\Rel_i$ is a subset of some $X\times Y$, 
where $X$ is either $2^\omega$, $\omega^\omega$ (or homeomorphic to it) or $\mathcal S$,
and $Y$ is the range of $\Rel_i$.
\begin{lemma}
For each $i$ and each $g\in Y$, the set
$\{f\in X:\, f\Rel_i g\}\subseteq X$ is meager.
\end{lemma}

\begin{proof}
We have explicitly defined each $f \Rel_i g$ as $\forall^*n R^n_i(f,g)$
for some $R^n_i$.
The lemma follows easily from the fact that 
for each $n\in\omega$,
the set $\{f\in X:\, R^n_i(f,g)\}$ is closed nowhere dense.
\end{proof}

\begin{lemma}\label{lem:coboundedunbounded}
Let $\lambda\le\kappa\le\mu$ be uncountable regular cardinals.
Force with $\mu$ many Cohen reals $(c_\alpha)_{\alpha\in \mu}$, followed by 
an $(\Rel_i,\lambda)$-good forcing.
Note that each Cohen real $c_\beta$ can be interpreted 
as element of the Polish space $X$ where $\Rel_i\subseteq X\times Y$. 
Then we get: For every real $r$ in the final extension's $Y$, the set 
$\{\alpha\in \kappa:\, c_\alpha\text{ is $\Rel_i$-unbounded by }r\}$ is cobounded in $\kappa$. I.e., 
$(\exists \alpha\in\kappa)\, (\forall \beta\in \kappa\setminus \alpha)\, \lnot c_\alpha\Rel_i r$.
\end{lemma}

\begin{proof}
Work in the intermediate extension after $\kappa$ many Cohen reals, let us call it $V_\kappa$.
The remaining forcing (i.e., $\mu\setminus\kappa$ many Cohens composed with the good forcing) is good; so applying the definition we get 
(in $V_\kappa$)
a set $\mathcal{Y}\subseteq Y$ of size ${<}\lambda$. 

As the initial Cohen extension is ccc, and $\kappa\ge \lambda$ is regular,
we get some $\alpha\in\kappa$ such that each element $y$ of $\mathcal{Y}$
already exists in the extension by the first $\alpha$ many Cohens, call it
$V_{\alpha}$.

Fix some $\beta\in\kappa\setminus \alpha$ and $y\in Y$. 
As  $\{x\in X:\ x\Rel_i y\}$ is a meager set already defined in $V_\alpha$, 
we get $\lnot c_\beta \Rel_i y$.
Accordingly, $c_\beta$ is unbounded by $\mathcal Y$; and, by the definition of good, 
unbounded by $r$ as well.
\end{proof}

In the light of this result, let us revisit Lemma~\ref{lem:connection} 
with some new notation, the ``linearly cofinally unbounded'' property $\mylin$:
\begin{definition}\label{def:linear}
For $i=1,2,3,4$, $\gamma$ a limit ordinal,  and $P$ a ccc forcing notion, let $\mylini(P,\gamma)$  stand for:
\begin{quote}
There is a sequence 
$(x_\alpha)_{\alpha\in\gamma}$ of $P$-names such that 
for every $P$-name $y$\\
$(\exists \alpha\in\gamma)\, (\forall \beta\in \gamma\setminus \alpha)\,P\forces \lnot x_\beta \Ri y)$.
\end{quote}
\end{definition}

\begin{lemma}\label{lem:linearcharacteristics}
\begin{itemize}
\item 
$\mylini(P,\delta)$
is equivalent 
to $\mylini(P,\cf(\delta))$.
\item
If  $\lambda$ is regular, then
$\mylini(P,\lambda)$ 
implies $\mathfrak{b}_i\le\lambda$ and $\mathfrak{d}_i\ge\lambda$. 
\end{itemize}
In particular: 
\begin{itemize}
\item[1.]
$\mylinI(P,\lambda)$ implies  $P\forces(\,\addnull\le\lambda\,\&\,\cofnull\ge\lambda\,)$.
\item[2.] 
$\mylinII(P,\lambda)$ implies  $P\forces(\,\mathfrak{b}\le\lambda\,\&\,\mathfrak{d}\ge\lambda\,)$.
\item[3.] 
$\mylinIII(P,\lambda)$ implies  $P\forces(\,\covnull\le\lambda\,\&\,\nonnull\ge\lambda\,)$.
\item[4.] 
$\mylinIV(P,\lambda)$ implies  $P\forces(\,\nonmeager\le\lambda\,\&\,\covmeager\ge\lambda\,)$.

\end{itemize}
\end{lemma}

\begin{proof} Assume that $(\alpha_\beta)_{\beta\in\cf(\delta)}$ is increasing continuous and cofinal
in $\delta$.
If $(x_\alpha)_{\alpha\in \delta}$
witnesses $\mylini(P,\delta)$,  then $(x_{\alpha_\beta})_{\beta\in\cf(\delta)}$ witnesses 
$\mylini(P,\cf(\delta))$. And 
if $(x_\beta)_{\beta\in\cf(\delta)}$ witnesses 
$\mylini(P,\cf(\delta))$, then  $(y_\alpha)_{\alpha\in\delta}$
witnesses $\mylini(P,\cf(\delta))$, where 
$y_{\alpha}\coloneq x_\beta$ for $\alpha\in [\alpha_\beta,\alpha_{\beta+1})$.

The set $\{x_\alpha:\, \alpha\in\lambda\}$ is certainly forced to be $\Rel_i$-unbounded;
and given a set $Y=\{y_j:\, j<\theta\}$ of $\theta<\lambda$ many $P$-names,
each has a bound $\alpha_j\in\lambda$ so that 
$(\forall \beta\in \lambda\setminus \alpha_j)\,P\forces \lnot x_\beta \Ri y_j)$, so for any $\beta\in\lambda$ above all $\alpha_j$ we get
$P\forces \lnot x_\beta \Ri y_j$ for all $j$; i.e., $Y$ cannot be dominating.
\end{proof}

\subsection{The initial forcing \texorpdfstring{$\Pa$}{P5} and the \texorpdfstring{$\mypart$}{COB} property}\label{ss:initalforcing}

We will assume the following throughout the paper:
\begin{assumption}\label{asm:P}
\begin{itemize}
    \item 
$\lambda_1<\lambda_2<\lambda_3<\lambda_4<\lambda_5$ are regular uncountable cardinals
such that $\mu<\lambda_i$
implies $\mu^{\aleph_0}<\lambda_i$. 
\item 
We set $\delta_5=\lambda_5+\lambda_5$, and 
partition $\delta_5\setminus \lambda_5$ into 
unbounded sets $S^i$
for $i=1,\dots,4$.
Fix for each $\alpha\in \delta_5\setminus \lambda_5$ a $w_\alpha\subseteq \alpha$ such that 
$\{w_\alpha:\, \alpha\in S^i\}$ is cofinal\footnote{i.e.,
if $\alpha\in S^i$ then $|w_\alpha|<\lambda_i$,
and for all $u\subseteq \delta_5$, $|u|<\lambda_i$
there is some $\alpha\in S^i$ with $w_\alpha\supseteq u$.} 
in $[\delta_5]^{{<}\lambda_i}$ (for each $i=1,\dots,4$).
\end{itemize}
\end{assumption}

The reader can assume that $(\lambda_i)_{i=1,\dots,5}$ and
$(S^i)_{i=1,\dots,4}$  
have been fixed once and 
for all (let us call them ``fixed parameters''), 
whereas we will investigate various possibilities for
$\bar w=(w_\alpha)_{\alpha\in \delta_5\setminus \lambda_5}$  in the following.
(We will call a $\bar w$ which satisfies the assumption a ``cofinal parameter''.)

We define by induction:
\begin{definition}\label{def:Pa}
We define the FS iteration $(P_\alpha, Q_\alpha)_{\alpha\in \delta_5}$ 
and, for $\alpha>\lambda_5$,
$P'_\alpha$  as follows: If $\alpha\in \lambda_5$, then $Q_\alpha$ is Cohen forcing. 
    In particular, the generic at $\alpha$ is determined by the Cohen real $\eta_\alpha$. For $\alpha\in \delta_5\setminus \lambda_5$: 
\begin{enumerate}
    \item
       $
            Q^\mathrm{full}_\alpha\defeq
            \left\{
              \begin{array}{c}
                \text{Amoeba}\\
                \text{Hechler}\\
                \text{Random}\\
                \QII\\
              \end{array}
            \right\}\text{ for $\alpha$ in}\left\{
              \begin{array}{l}
                S^1\\
                S^2\\
                S^3\\
                S^4\\  
              \end{array}
            \right.
        $.
    \\
    So $Q^\mathrm{full}_\alpha$
    is a Borel definable subset of the 
    reals, and the $Q^\mathrm{full}_\alpha$-generic 
    is determined, in a Borel way, by
    the canonical generic real $\eta_\alpha$.
    \item\label{item:asm} $P'_\alpha$ is the set of conditions $p\in P_\alpha$
    satisfying the following, for each 
    $\beta\in\supp(p)$:
    $\beta\in w_\alpha$ and there
    is (in the ground model) 
    a countable $u\subseteq w_\alpha\cap \beta$ and a Borel function
    $B:(\omega^\omega)^{u}\to Q^\mathrm{full}_\beta$ such that
    $p\restriction\beta$ forces
    that $p(\beta)=B((\eta_\gamma)_{\gamma\in u})$.   
    We \emph{assume} that 
    \begin{equation}\label{eq:cpl}
    P'_\alpha\text{ is a complete subforcing of }P_\alpha. 
    \end{equation}
    \item\label{item:asm44} In the $P_\alpha$-extension, let $M_\alpha$ be the induced $P'_\alpha$-extension of $V$.
        Then $Q_\alpha$ is the $M_\alpha$-evaluation of $Q^\mathrm{full}_\alpha$. Or equivalently (by absoluteness): 
        $Q_\alpha= Q^\mathrm{full}_\alpha\cap M_\alpha$.
        We call $Q_\alpha$ a ``partial $Q^\mathrm{full}_\alpha$ forcing'' (e.g.:
        a ``partial random forcing'').
        \end{enumerate}
\end{definition}
Some notes:
\begin{itemize}
    \item 
    For item~(\ref{item:asm44}) to make sense,~\eqref{eq:cpl} is required.
    \item We do not require any ``transitivity'' of the $w_\alpha$, i.e., 
    $\beta\in w_\alpha$ does generally not imply $w_\beta\subseteq w_\alpha$.
    \item We do not
require (and it will generally not be true)
that $P_\alpha$ forces that $Q_\alpha$
is a \emph{complete} subforcing of $Q^\mathrm{full}_\alpha$.
\end{itemize}

A simple absoluteness argument (between $M_\alpha$
and $V[G_\alpha]$) shows:
\begin{lemma}\label{lem:iterationbasic}
$P_\alpha$ forces:
\begin{enumerate}
    \item[(a)]  $Q_\alpha$ is an incompatibility preserving subforcing of $Q^\mathrm{full}_\alpha$ and in particular ccc.
    (And so, $P_\alpha$ itself is ccc for all $\alpha$.)
    \item[(b)] For $\alpha\in S^i$, $|Q_\alpha|<\lambda_i$.
    \item[(c)] $Q_\alpha$ forces that its generic filter $G(\alpha)$ is also generic 
    over $M_\alpha$. So from the point
    of view of $M_\alpha$,
    $M_\alpha[G(\alpha)]$ is a $Q^\mathrm{full}_\alpha$-extention.
    \item[(2)] For $\alpha\in S^2$:
    The partial Hechler forcing $Q_\alpha$ is $\sigma$-centered.
    \item[(3)] For $\alpha\in S^3$: The partial random forcing
    $Q_\alpha$ equivalent to a subalgebra of the random algebra.
    \item[(4)] For $\alpha\in S^4$: 
    A partial $\QII$ forcing is $(\rho,\pi)$-linked 
    and basically equivalent to a subalgebra of the random algebra (as in Lemma~\ref{lem:QIIbasic}(\ref{item:subrand})).
\end{enumerate}
\end{lemma}

\begin{proof}
(b): $|P_\alpha'|\le |w_\alpha|^{\aleph_0}\times 2^{\aleph_0}<\lambda_i$ by Assumption~\ref{asm:P}.
There is a set of nice $P_\alpha'$-names
of size $<\lambda_i$ such that 
every 
$P_\alpha'$-name for a real
has an equivalent name in this set.
Accordingly, the size of the reals in $M_\alpha$
is forced to be $<\lambda_i$.

(c) is trivial, as $Q_\alpha$ 
is element of the transitive class $M_\alpha$.

(4): By Lemma~\ref{lem:QIIbasic}(\ref{item:pr}) we know that $M_\alpha$ thinks that $\QII$
is $(\rho, \pi)$-linked; i.e., that there 
is a family\footnote{Actually there is even a Borel definable family $Q^i_j$, see the proof of  Lemma~\ref{lem:QIIbasic}(\ref{item:linked}), but this is not required here.}  $Q^i_j$ as in  Definition~\ref{def:pirholinked}.
Being $\ell$-linked is obviously absolute between
$M_\alpha$ and $V[G_\alpha]$ (for any $\ell<\omega$);
and $M_\alpha\vDash \bigcup_{h\in\omega,i<\rho(h)}Q^h_i=Q^\mathrm{full}_\alpha$
translates to $V[G_\alpha]\vDash \bigcup_{h\in\omega,i<\rho(h)}Q^h_i=Q_\alpha$.

Similarly, $M_\alpha$ thinks that $\QII$ 
satisfies~\ref{lem:QIIbasic}(\ref{item:subrand}),
i.e., that there is some dense $Q'\subseteq \QII$ and a dense embedding from $Q'$ to
a subalgebra $B'$ of the random algebra.

So from the point of view of
$V[G_\alpha]$, there is a $Q'$ dense 
in $\QII\cap M_\alpha$ and
a dense embedding of $Q'$ into some $B'$,
which is a subalgebra of
the random algebra in $M_\alpha$
and therefore of the random algebra
in $V[G_\alpha]$.
\end{proof}

It is easy to see that~\eqref{eq:cpl} is a ``closure property'' of $w_\alpha$:
\begin{lemma}\label{lem:closure1}
Assume we have constructed (in the ground model) 
$(P_\beta,Q_\beta)_{\beta<\alpha}$ and $w_\alpha$
according to Definition~\ref{def:Pa};
for some
$\alpha\in S^i$,  $i=1,\dots,4$.
This determines the (limit or composition) $P_\alpha$.
\begin{enumerate}[(a)]
    \item\label{item:borel} For every $P_\alpha$-name $\tau$ of a real,
    there is (in $V$) a countable $u\subseteq \alpha$
    and a Borel function $B:(\omega^\omega)^u\to \omega^\omega$
such that $P_\alpha$ forces $\tau=B((\eta_\gamma)_{\gamma\in u})$.

(So if $w_\alpha\supseteq u$ satisfies~\eqref{eq:cpl},
then $P_\alpha$ forces that $\tau\in M_\alpha$.)
    \item\label{item:closure1}
    The set of $w_\alpha$ satisfying~\eqref{eq:cpl} is an $\omega_1$-club in $[\alpha]^{<\lambda_i}$ (in the ground model).
\end{enumerate}
\end{lemma}
(A set $A\subseteq [\alpha]^{{<}\lambda_i}$ is an $\omega_1$-club, if for each
$a\in [\alpha]^{{<}\lambda_i}$ there is a $b\supseteq a$ in $A$, and if $(a^i)_{i\in\omega_1}$ is an increasing sequence of sets in $A$, then the limit $b\coloneq \bigcup_{i\in\omega_1}a^i$ is in $A$ as well.)

\begin{proof}
The first item follows easily from the fact that 
we are dealing with a FS ccc iteration
where the generics of all iterands $Q_\beta$ are Borel-determined by some generic real $\eta_\beta$.
(See, e.g., \cite[1.2]{eight}, for more details.)

Any $w\in [\alpha]^{<\lambda_i}$
defines some $P^w_\alpha$.
We first define $w'$ for such a $w$:

Set $X=[P^w_\alpha]^{\le\aleph_0}$, 
as set
of size at most 
$(2^{\aleph_0}\times |w|^{\aleph_0})^{\aleph_0}<\lambda_i$.
For $x\in X$, pick some
$p\in P_\alpha$ 
stronger than all conditions in $x$
(if such a condition exists),
and some $q\in P_\alpha$ 
incompatible to each element of $x$
(again, if possible).
There is a countable $w_x\subseteq\alpha$
such that $p,q\in P^{w_x}$.
Set $w'\defeq w\cup \bigcup_{x\in X} w_x$.

Start with any $w_0\in [\alpha]^{<\lambda_i}$.
Construct an increasing continuous chain in $[\alpha]^{<\lambda_i}$ with $w^{k+1}=(w^k)'$.
Then $w^{\omega_1}\supseteq w_0$ 
is in the set of $w$ satisfying~\eqref{eq:cpl}; 
which shows that this set is 
unbounded.
It is equally easy to see that it is 
closed under increasing sequences of length $\omega_1$.
\end{proof}

For later reference, we explicitly state the assumption we used (for every
$\alpha\in\delta_5\setminus \lambda_5$):
\begin{assumption}\label{asm:completesubforcing}
$w_\alpha$ is sufficiently closed so that \eqref{eq:cpl} is satisfied.
\end{assumption}

Let us also restate Lemma~\ref{lem:closure1}(a):
\begin{lemma}\label{lem:jkwhrwe924}
For each $\Pa$-name
$f$ of a real, there is a countable
set $u\subseteq \delta_5$ such that
$w_\alpha\supseteq u$ implies that ($\Pa$ forces that) $f\in M_\alpha$.
\end{lemma}
%

\begin{lemma}\label{lem:linearPapartial}
$\mylini(\Pa,\kappa)$ holds for $i=1,3,4$ and each regular cardinal $\kappa$ in $[\lambda_i,\lambda_5]$.
\end{lemma}

\begin{proof}
This follows from Lemma~\ref{lem:iterationbasic}:

For $i=1$: 
Partial random and partial $\QII$ forcings are
basically equivalent to a sub-Boolean-algebra of
the random algebra;
and partial Hechler forcings are $\sigma$-centered.
The partial amoeba forcings are small, i.e., have size ${<}\lambda_1$.
So according to 
Lemma~\ref{lem:gettinggood}, all iterands $Q_\alpha$ (and therefore the limits as well)
are $(\RI,\lambda_1)$-good.

For $i=3$, note that partial $\QII$ forcings are $(\rho,\pi)$-linked.
All other iterands have size ${<}\lambda_3$,
so the forcing is $(\RIII,\lambda_3)$-good.

For $i=4$ it is enough to note that \emph{all} iterands are small, i.e., of size
${<}\lambda_4$.

We can now apply Lemma~\ref{lem:coboundedunbounded}.
\end{proof}

So in particular, $\Pa$ forces $\addnull\le\lambda_1$, $\covnull\le\lambda_3$, $\nonmeager\le\lambda_4$ and $\covmeager=\nonnull=\cofnull=\lambda_5=2^{\aleph_0}$;
i.e., the respective left hand characteristics are small. 
We now show that they are also large,  using the ``cone of bounds'' property $\mypart$:

%

\begin{definition}\label{def:partial}
For a ccc forcing notion $P$, regular uncountable cardinals $\lambda,\mu$
and 
$i=1,2,4$, let $\myparti(P,\lambda,\mu)$ stand for:
\begin{quote}
There is a ${<}\lambda$-directed partial order $(S,\prec)$
of size $\mu$
and a sequence $(g_s)_{s\in S}$ of $P$-names for reals
such that for each $P$-name $f$ of a real
\\
$(\exists s\in S)\,(\forall t\succ s)\, P\forces f \Ri g_t $.
\end{quote}
For $i=3$, let $\mypartIII(P,\lambda,\mu)$ stand for:
\begin{quote}
There is a ${<}\lambda$-directed partial order $(S,\prec)$
of size $\mu$
and a sequence $(g_s)_{s\in S}$ of $P$-names for reals
such that for each $P$-name $f$ of a null-set 
\\
$(\exists s\in S)\,(\forall t\succ s)\, P\forces g_t\notin f$.
\end{quote}
\end{definition}
So $s$ is the tip of a cone that consists of elements bounding $f$, where in case $i=3$ we implicitly use an additional relation
$N \RIII' r $ expressing that the null-set $N$ doesn't contain the real $r$.
Note that $\covnull$ is the bounding number $\mathfrak b'_3$ of $\RIII'$, and $\nonnull$ the dominating number $\mathfrak d'_3$.
So $\addnull=\mathfrak b'_3\le\mathfrak b_3$ and $\nonnull=\mathfrak d'_3\ge \mathfrak d_3$
(as defined 
in Lemma~\ref{lem:connection}).

$\mypart_i(P,\lambda,\mu)$ implies that $P$ forces that $\mathfrak b_i\ge \lambda$
 and that $\mathfrak d_i\le \mu$ (for $i=1,2,4$, and the same for 
 $i=3$ and $\mathfrak b'_3$, $\mathfrak d'_3$):
Clearly $P$ forces that $\{g_s:\, s\in\mathcal S\}$ is dominating.
And if $A$ is set of names of size $\kappa<\lambda$, then 
for each $f\in A$ the definition gives a bound $s(f)$ and directedness some 
$t\succ s(f)$ for all $f$, i.e., $g_t$ bounds all elements of $A$.
So we get:
\begin{lemma}\label{lem:partialcharacteristics}

\begin{itemize}
\item[1.] $\mypartI(P,\lambda,\mu)$ implies $P\forces(\, \addnull\ge\lambda \,\&\, \cofnull\le \mu\,)$.
\item[2.] $\mypartII(P,\lambda,\mu)$ implies $P\forces(\, \mathfrak{b}\ge\lambda \,\&\, \mathfrak{d}\le \mu\,)$.
\item[3.] $\mypartIII(P,\lambda,\mu)$ implies $P\forces(\, \covnull\ge\lambda \,\&\, \nonnull\le \mu\,)$.
\item[4.] $\mypartIV(P,\lambda,\mu)$ implies $P\forces(\, \nonmeager\ge\lambda \,\&\, \covmeager\le \mu\,)$.
\end{itemize}
\end{lemma}

\begin{lemma}\label{lem:partialPa}
$\myparti(\Pa,\lambda_i,\lambda_5)$ holds
(for $i=1,2,3,4$).
\end{lemma}

\begin{proof}
We use the following facts (provable in ZFC, or true in the $P_\alpha$-extention, respectively):
\begin{itemize}
\item[1.]
Amoeba forcing adds a sequence $\bar b$ which $\RI$-dominates the old elements of $\mathcal{C}$.

(The simple proof can be found 
in~\cite[Lem.~1.4]{ten}, a slight variation in~\cite{BJ}.)

Accordingly (by absoluteness), 
the generic real $\eta_\alpha$ 
for partial amoeba forcing $Q_\alpha$ $\RI$-dominates $\mathcal{C}\cap M_\alpha$.

\item[2.]
Hechler forcing adds a real which $\RII$-dominates all old reals.

Accordingly, the 
generic real $\eta_\alpha$ for partial Hechler forcing $Q_\alpha$
$\RII$-dominates all reals in $M_\alpha$.

\item[3.] Random forcing adds a random real.

Accordingly, the generic real $\eta_\alpha$ for partial random forcing 
$Q_\alpha$ is not in any nullset
whose Borel-code is in $M_\alpha$.
\item[4.]
The generic branch $\eta\in\lim(T^*)$ added by
$\QII$ is eventually different to each old real, i.e., 
$\RIV$-dominates the old reals.

(This was shown in Lemma~\ref{lem:QIIbasic}(\ref{item:ed}).)

Accordingly, the generic branch $\eta_\alpha$ for partial $\QII$ forcing $Q_\alpha$
$\RIV$-dominates the reals in $M_\alpha$.
\end{itemize}
Fix $i\in\{1,2,3,4\}$, and set $S=S^i$ and 
$s\prec t$ if $w_s\subsetneq w_t$, and  
let $g_s$ be $\eta_s$, i.e., the generic added at $s$ (e.g., the partial random real in case of $i=3$, etc).

Fix a $\Pa$-name $f$ for a real.
It depends (in a Borel way) on 
a countable index set $w^*\subseteq \delta_5$.
Fix some $s\in S^i$ such that $w_s\supseteq w^*$.
Pick any $t\succ s$. Then $w_t\supseteq w_s\supseteq w^*$, so ($\Pa$ forces that)
$f\in M_t$, so, as just argued,
$\Pa\forces f \Ri g_t $
(or: $\Pa\forces f \RIII' g_t $ for $i=3$). 
\end{proof}

So to summarize what we know so far about $\Pa$:
Whenever we choose (in addition to the ``fixed'' $\lambda_i$, $S^i$)
a cofinal parameter $\bar w$ satisfying Assumptions~\ref{asm:P} and~\ref{asm:completesubforcing}, we get
\begin{fact}\label{fact:summary}
\begin{itemize}
\item $\myparti$ holds for $i=1,2,3,4$. So the left hand side characteristics are large.
\item $\mylini$ holds for $i=1,3,4$. So the left hand side characteristics other than $\mathfrak{b}$ are small.
\end{itemize}
\end{fact}
What is missing is ``$\mathfrak{b}$ small''. We do not claim that this
will be forced for every $\bar w$ as above; but we will show in the 
rest of Section~\ref{sec:partA}
that we can choose such a $\bar w$.

\subsection{FAMs in the \texorpdfstring{$P_\alpha$}{P alpha}-extension compatible with \texorpdfstring{$M_\alpha$}{M alpha}, explicit conditions.}
We first investigate sequences 
$\bar q=(q_\ell)_{\ell\in\omega}$ of
$Q_\alpha$-conditions that are in $M_\alpha$,
i.e., the (evaluations of) $P'_\alpha$-names
for $\omega$-sequences in 
$Q^\mathrm{full}_\alpha$.
For $\alpha\in S^3\cup S^4$,
$M_\alpha$ thinks that $Q_\alpha$ (i.e., $Q^\mathrm{full}_\alpha$)
has FAM-limits.
So if $M_\alpha$ thinks that $\Xi_0$ is a FAM,
then for any sequence $\bar q$ in $M_\alpha$ there is a
condition $\lim_{\Xi_0}(\bar q)$ in $M_\alpha$ (and thus in $Q_\alpha$).
We can relativize 
Lemma~\ref{lem:extension} to sequences in $M_\alpha$:
\begin{lemma}\label{lem:times}
Assume that $\alpha\in S^3\cup S^4$, that
$\Xi$ is a $P_\alpha$-name for a 
FAM and that
$\Xi_0$, the restriction of $\Xi$ to $M_\alpha$, is forced to be in
$M_\alpha$.
Then there is a $P_{\alpha+1}$-name $\Xi^+$ for a FAM
such that
for all $(\stem^*,\loss^*)$-sequences $\bar q$ in $M_\alpha$,
\[
\lim\nolimits_{\Xi_0}(\bar q)\in G(\alpha)\text{ implies }\Xi^+(A_{\bar q})\ge 1-\sqrt{\loss^*}.
\]
\end{lemma}
$A_{\bar q}$ was defined in~\eqref{eq:ajkrhw}
(here we use $G(\alpha)$ instead of $G$, of course).

\begin{proof} 
This Lemma is implicitly used in \cite{sh}.
Note that $P_\alpha'$ is a complete
subforcing of $P_\alpha$, 
and so there is a quotient $R$
such that $P_\alpha=P_\alpha'*R$.
We consider the following (commuting) diagram:
\[
\xymatrix@=2.5ex{
V\ar[r]^{P_\alpha}\ar[dr]_{P'_\alpha} & V_\alpha\ar[r]^{Q_\alpha}& V_{\alpha+1} \\
& M_\alpha\ar[u]_R\ar[r]_{Q_\alpha} &\myex\ar[u]\\
}
\]
%
Note that ($P'_\alpha$ forces that)
$R*Q_\alpha=R\times Q_\alpha$.
So from the point of view of $M_\alpha$:
\begin{itemize}
\item 
$Q_\alpha=Q^\mathrm{full}_\alpha$ has FAM limits,
and $\Xi_0$ is a FAM.
So there is a $Q_\alpha$-name for 
a FAM $\Xi_0^+$ satisfying Lemma~\ref{lem:extension}.
\item
$R$ is a ccc forcing, and there is
an $R$-name\footnote{We identify the $P_\alpha$-name $\Xi$ in 
$V$ and the induced $R$-name in $M_\alpha=V[G'_\alpha]$.}
$\Xi$ for a FAM extending $\Xi_0$.
\item So 
there is $R\times Q_\alpha$-name 
    $\Xi^+$ for a FAM extending both $\Xi^+_0$ and $\Xi$
    (cf.~\cite[Claim 1.6]{sh}).
\end{itemize}
Back in $V$, this defines the $P_{\alpha+1}$-name $\Xi^+$.
Let $\bar q=(q_\ell)_{\ell\in \omega}$ be a sequence in $M_\alpha$.
Then $M_\alpha[G(\alpha)]$ thinks:
If $\lim_{\Xi_0}(\bar q)\in G(\alpha)$, then $\Xi_0^+(A_{\bar q})$ 
is large enough.
This is upwards absolute to $V[G_{\alpha+1}]$
(as $A_{\bar q}$ is absolute).
\end{proof}

For later reference, we will reformulate the lemma for a specific instance
of ``sequence in $M_\alpha$''. Recall that a 
sequence in $M_\alpha$ corresponds to a
``$P'_\alpha$-name of a sequence in $Q^\mathrm{full}_\alpha$''.
This is not equivalent to a
``$P_\alpha$-name for a sequence in $Q_\alpha$'',
which would correspond to an arbitrary sequence in $Q_\alpha$
(of which there are $|\alpha+\aleph_0|^{\aleph_0}$ many,
while there are only less than $\lambda_i$ many sequences in $M_\alpha$).
However, we can define the following:
\begin{definition}\label{def:explicit}
\begin{itemize}
    \item 
An explicit $Q_\alpha$-condition (in $V$) is a $P'_\alpha$-name for 
a $Q^\mathrm{full}_\alpha$ condition.
\item
A condition $p\in \Pa$ is explicit, if for all $\alpha\in\supp(p)\cap(S^4\cup S^5)$,
$p(\alpha)$ is an explicit $Q_\alpha$-condition.
\end{itemize}
\end{definition}
Here we mean that for $p(\alpha)$ there is 
a $P'_\alpha$-name $q_\alpha$ such that $p\restriction \alpha\Vdash p(\alpha)=q_\alpha$ (and the map $\alpha\mapsto q_\alpha$ exists in the ground model, i.e., we do not just have a $P_\alpha$-name for a $P'_\alpha$-condition $q_\alpha$).
\begin{lemma}
The set of explicit conditions is dense.
\end{lemma}

\begin{proof}
We show by induction that the set $D_\alpha$ of explicit conditions in
$P_\alpha$ is dense in $P_\alpha$.
As we are dealing with FS iterations, limits are clear.
Assume that $(p,q)\in P_{\alpha+1}$.
Then $p$ forces that there is a $P'_\alpha$-name $q'$ 
such that $q'=q$. Strengthen $p$ to some $p'\in D_\alpha$ deciding $q'$.
Then $(p',q')\le (p,q)$ is explicit.
\end{proof}

Note that any sequence in $V$ of explicit $Q_\alpha$-conditions
defines a sequence of conditions in $M_\alpha$ (as $V\subseteq M_\alpha$). 
So we get:
\begin{lemma}\label{lem:explicit}
Let $\alpha$, $\Xi$, and $\Xi^+$ be as in Lemma~\ref{lem:times}, and
let $(p_\ell)_{\ell\in\omega}$ be (in $V$) 
a sequence of explicit conditions in $\Pa$
such that 
$\alpha\in\supp(p_\ell)$
for all $\ell\in\omega$. Set $q_\ell\defeq p_\ell(\alpha)$
and $\bar q\defeq (q_\ell)_{\ell\in\omega}$, and assume
that $(\stem(q_\ell), \loss(q_\ell))$ is forced to be equal to some
constant $(\stem^*,\loss^*)$. 

Then there is a $P'_\alpha$-name for a $Q^\mathrm{full}_\alpha$-condition 
(and thus a $P_\alpha$-name for a $Q_\alpha$-condition)
$\lim_{\Xi_0}(\bar q)$
such that 
$\lim_{\Xi_0}(\bar q)$
forces that $\Xi^+(A_{\bar q})\le 1-\sqrt{\loss^*}$.
\end{lemma}

\subsection{Dealing with \texorpdfstring{$\mathfrak{b}$}{b} (without GCH)}
In this section, we follow~\cite[1.3]{ten}, additionally using techniques inspired by~\cite{sh}.

We assume  the following
(in addition to Assumption~\ref{asm:P}):
\begin{assumption}\label{asm:chi}(This section only.) 
$\chi<\lambda_3$ is regular such that $\chi^{\aleph_0}=\chi$,  
$\chi^+\ge \lambda_2$ and
$2^\chi =|\delta_5|= \lambda_5$.
\end{assumption}

Set $S^0=\lambda_5\cup S^1\cup S^2$.
So $\delta_5=S^0\cup S^3\cup S^4$,
and $\Pa$ is a FS ccc iteration along $\delta_5$ such that $\alpha\in S^0$
implies $|Q_\alpha|<\lambda_2$, i.e., $|Q_\alpha|\le\chi$
(and $Q_\alpha$ is a partial random forcing for $\alpha\in S^3$ and a partial
$\QII$-forcing for $\alpha\in S^4$).

Let us fix, for each $\alpha\in S^0$, a $P_\alpha$-name 
\begin{equation}\label{eq:ia}
i_\alpha:Q_\alpha\to\chi\text{ injective}.
\end{equation}

%

\begin{definition}
\begin{itemize}
   \item A ``partial guardrail'' is a 
function $h$ defined on a subset of $\delta_5$  
such that, for $\alpha\in\dom(h)$:
$h(\alpha)\in \chi$ if $\alpha\in S^0$; and $h(\alpha)$
is a pair $(x,y)$ with $x\in H(\aleph_0)$ and $y$ a rational number otherwise.
(Any ($\stem,\loss$)-pair is of this form.) 
%
  \item A ``countable guardrail'' is a partial guardrail with
  countable domain. A ``full guardrail'' is a partial guardrail with domain $\delta_5$.
\end{itemize}
\end{definition}  

We will use the following lemma, which is a consequence of the Engelking-Karlowicz theorem~\cite{MR0196693} on the density of box products (cf.~\cite[5.1]{MR3513558}):  
\begin{lemma}\label{use.EK}
(As $|\delta_5|\le 2^\chi$.) There is a family $H^*$ of full guardrails
of cardinality $\chi$ such that each countable guardrail is extended by
some $h\in H^*$. 
We will fix such an $H^*$.
\end{lemma}

Note that the notion of guardrail (and the density property required in
Lemma~\ref{use.EK}) only
depends on the ``fixed'' parameters 
$\chi$, $\delta_5$, $S^0$, $S^3$ and $S^4$;
so we can fix an $H^*$ that will work for all these fixed parameters
and all choices of the cofinal parameter $\bar w$.

Once we have decided on $\bar w$, and thus have defined $\Pa$, we can define the following:
\begin{definition}
$D^*\subseteq \Pa$ consists of $p$ such that there is a 
partial guardrail $h$ (and we say: 
``$p$ follows $h$'') with $\dom(h)\supseteq \supp(p)$ and, 
for all $\alpha\in\supp(p)$, 
\begin{itemize}
\item If $\alpha\in S^0$,
then $p\restriction\alpha\Vdash i_\alpha(p(\alpha))=h_\alpha$.
\item If $\alpha\in S^3\cup S^4$, the empty condition of $P_\alpha$ forces
\[
p(\alpha)\in Q_\alpha\text{ and }(\stem(p(\alpha)),\loss(p(\alpha)))=h(\alpha).
\]
\item Furthermore, $\sum_{\alpha\in \supp(p)\cap(S^3 \cup S^4)}\sqrt{\loss(p(\alpha))}<\frac12$.
\item $p$ is explicit (as in Definition~\ref{def:explicit}).
\end{itemize}

%
\end{definition}

\begin{lemma}
$D^*\subseteq \Pa$ is dense.
\end{lemma}

\begin{proof}
By induction we show that for any sequence $(\epsilon_i)_{i\in\omega}$
of positive numbers
the following set of $p$ is dense: 
If $\supp(p)=\{\alpha_0,\dots,\alpha_m\}$, where $\alpha_0>\alpha_1>,\dots$
(i.e., we enumerate downwards),
$\loss^p_{\alpha_n}<\epsilon_n$ whenever $\alpha_n\in S^3\cup S^4$.
For the successor step, we use that the set of $q\in Q_\alpha$
such that $\loss(q)<\epsilon_0$ is forced to be dense.
\end{proof}

\begin{remark}
So the set of conditions following \emph{some} guardrail  is dense.
For each \emph{fixed} guardrail $h$, the set      
of all conditions $p$  following $h$ is 
$n$-linked, 
provided that each $\loss$ in the domain of $h$
is $<\frac1n$ (cf.\ Assumption~\ref{asm:qlim}).
\end{remark}

\begin{definition}
  A ``$\Delta$-system with heart $\nabla$ following the guardrail $h$'' is a family $\bar p=(p_i)_{i\in I}$ 
     of conditions such that
     \begin{itemize}
       \item all $p_i$ are in $D^*$ and follow $h$, 
       \item 
       $(\supp(p_i))_{i\in I}$ is a $\Delta$ system with heart $\nabla$ in the usual sense (so $\nabla\subseteq \delta_5$ is finite)
       \item        
       the following is independent of $i\in I$:
         \begin{itemize}
           \item $|\supp(p_i)|$, which we call $m^{\bar p}$. 
           
           Let $(\alpha^{\bar p,n}_i)_{n<m^{\bar p}}$ increasingly enumerate $\supp(p_i)$.
           \item Whether $\alpha^{\bar p,n}_i$ is less than, equal to or bigger than the $k$-th element of $\nabla$. 
           
           In particular it is independent of $i$ whether $\alpha^{\bar p,n}_i\in\nabla$, in which case we call $n$ a  ``heart position''.
           \item Whether $\alpha^{\bar p,n}_i$ is in $S^0$, in $S^3$ or in $S^4$.
           
           If  $\alpha^{\bar p,n}_i\in S^j$, we call $n$ an ``$S^j$-position''.
           \item If $n$ is not an $S^0$-position:\footnote{If $n$ is a $S^0$-position, $h(\alpha^{\bar p,n}_i)$ will generally not be be independent of $i$; unless of course $n$ is a heart position.}           
            The value of $h(\alpha^{\bar p,n}_i)\eqdef (\stem^{\bar p,n},\loss^{\bar p,n})$. If $n$ is an $S^0$-position, we set $\loss^{\bar p,n}\defeq 0$.  
       \end{itemize}              
       \end{itemize}   
A ``countable $\Delta$-system''
$\bar p=(p_\ell:\ell\in \omega)$
is a $\Delta$ system that additionally satisfies:
        \begin{itemize}
           \item For each non-heart position\footnote{For a heart position $n$, $(\alpha^{\bar p,n}_\ell)_{\ell\in\omega}$ is of course constant.} $n<m^{\bar p}$, the sequence $(\alpha^{\bar p,n}_\ell)_{\ell\in\omega}$ is strictly increasing.
       \end{itemize}   
\end{definition}


\begin{fact}
\begin{itemize}
\item   Each infinite $\Delta$-system $(p_i)_{i\in I}$ contains a countable $\Delta$-system. I.e., there is a sequence $i_\ell$ in $I$ such that
$(p_{i_\ell})_{\ell\in \omega}$ is a countable $\Delta$-system..
\item
If $\bar p$ is a $\Delta$-system (or: a countable $\Delta$-system)
  following $h$ with heart $\nabla$, and $\beta\in\nabla\cup (\max(\nabla+1))$, then $\bar p\restriction\beta\defeq (p_i\restriction\beta)_{i\in I}$    
  is again a $\Delta$-system (or: a countable $\Delta$-system, respectively) following $h$, now with heart $\nabla\cap \beta$.
\end{itemize}
\end{fact}

\begin{definition}\label{def:Mlimit}
Let $\bar p$ be a 
   countable 
   $\Delta$-system, 
   and assume that $\bar
   \Xi=(\Xi_\alpha)_{\alpha\in \nabla \cap (S^3 \cup S^4)}$ is  
   a sequence 
   such that each $\Xi_\alpha$ is a 
   $P_\alpha$-name for a FAM and $P_\alpha$ forces that 
   $\Xi_\alpha$ restricted to $M_\alpha$ is in $M_\alpha$.
   Then we can define $q=\lim_{\bar \Xi}(\bar p)$
   to be the following $\Pa$-condition with support $\nabla$: 
   \begin{itemize}
   \item If $\alpha\in \nabla\cap S^0 $, then $q(\alpha)$ is the common 
   value of all $p_n(\alpha)$.  (Recall that this value is already determined by 
    the guardrail $h$.)
   \item If $\alpha\in \nabla\cap (S^3 \cup S^4)$, then $q(\alpha)$ is
    (forced by $\Pa_\alpha$ to be) 
    $\lim_{\Xi_\alpha}(p_\ell(\alpha))_{\ell\in \omega}$,
    see Lemma~\ref{lem:explicit}.
   \end{itemize}
\end{definition}


We now  give a 
specific way to construct such $\bar w$, which allows to keep
$\mathfrak b$ small.

\begin{construction}\label{constr}
We can construct by induction on $\alpha\in \delta_5$ for each
$h\in H^*$ some $\Xi^{h}_\alpha$, and, if $\alpha>\kappa_5$, also $w_\alpha$, such that:
\begin{enumerate}[(a)]
\item Each $\Xi^{h}_\alpha$ is a $P_\alpha$-name of a FAM extending 
 $\bigcup_{\beta<\alpha}\Xi^{h}_\beta$.
\item If $\alpha$ is a limit of countable cofinality:
Assume $\bar p$ is a countable $\Delta$-system  in $P_\alpha$ 
following $h$,
and $n<m^{\bar p}$ such that 
$(\alpha^{\bar p,n}_\ell)_{\ell\in\omega}$ has supremum $\alpha$.
Then
$A_{\bar p,n}$ is forced to have $\Xi^{h}_\alpha$-measure $1$, where 
\[
A_{\bar p,n}\coloneq \Big\{k\in\omega:\, \Big|\big\{\ell\in I_k:\, 
p_\ell(\alpha^{\bar p,n}_\ell)\in G(\alpha^{\bar p,n}_\ell)\big\}\Big|\ge |I_k|\cdot\big(1-\sqrt{\loss^{\bar p,n}}\,\big) \Big\}
\]

\item 
For each countable $\Delta$-system $\bar p$ in $P_\alpha$ following $h$, 
the $P_\alpha$-condition
$\lim_{(\Xi^{h}_\beta)_{\beta<\alpha}}(\bar p)$ is well-defined 
and forces  
\begin{multline*}
\Xi^{h}_\alpha(A_{\bar p})\ge 
1-\sum_{n<m^{\bar p}}\sqrt{\loss^{\bar p,n}},
 \text{ where}\\
A_{\bar p}\coloneq \Big\{k\in\omega:\, \left|\big\{\ell\in I_k:\, 
p_\ell\in G_\alpha\big\}\right|\ge |I_k|\cdot \big(1-\sum_{n<m^{\bar p}}\sqrt{\loss^{\bar p,n}}\,\big)
\Big\}.
\end{multline*}
\item For $\alpha>\kappa_5$:
$w_\alpha$ is ``sufficiently closed''.
More specifically: It satisfies Assumptions~\ref{asm:P} and~\ref{asm:completesubforcing}, and 
if $\alpha\in S^3\cup S^4$ then
$P_\alpha$ forces that $\Xi^h_\alpha$
 restricted to $M_\alpha$ is in $M_\alpha$.

Actually, the set of $w_\alpha$ satisfying this
is an $\omega_1$-club set.
 
%
\end{enumerate}
\end{construction}
%
%
\begin{proof}
\textbf{\boldmath (a\&c) for $\cf(\alpha)>\omega$:\/}
We set $\Xi^{h}_\alpha=\bigcup_{\beta<\alpha} \Xi^{h}_\beta$.
As there are no new reals at uncountable confinalities, this is a FAM.
Each countable $\Delta$-system is bounded by some 
$\beta<\alpha$, and, by induction, (c) holds for $\beta$;  so 
(c) holds for $\alpha$ as well.

\textbf{\boldmath (a\&b) for $\cf(\alpha)=\omega$:\/}
Fix $h$.
We will show that $P_\alpha$
forces 
$A\cap \bigcap_{j<j^*}  A_{\bar p^j,n^j}\neq \emptyset$,
where $A$ is a $\Xi^{h}_\beta$-positive set for some $\beta<\alpha$,
and each $(\bar p^j,n^j)$ is as in (b).

Then we can work in the $P_\alpha$-extension and 
apply Fact~\ref{fact:FAMextensions}(\ref{item:uf}),
using $\bigcup_{\beta<\alpha}\Xi^{h}_\beta$ as
the partial FAM $\Xi'$. This gives an extension of $\Xi'$ to a FAM 
$\Xi^h_{\alpha}$
that assigns measure one to all $A_{\bar p,n}$, showing that 
(a) and (b) are satisfied.

So assume towards a contradiction that some $p\in P_\alpha$ forces 
\[ A\cap \bigcap_{j<j^*}  A_{\bar p^j,n^j}=\emptyset.\]

We can assume that $p$ decides the $\beta$ such that $A\in V_\beta$,
that $\beta$ is above the hearts of all $\Delta$-sequences 
$\bar p^j$ involved, and that 
$\supp(p)\subseteq \beta$. 
We can extend $p$ to some $p^*\in P_\beta$ to decide $k\in A$
for some ``large'' $k$:
By large, we mean:
\begin{itemize}
\item 
Let $F(l;n,p)$ (the cumulative binomial probability distribution)
be the probability that $n$ independent 
experiments, each with success probability $p$, 
will have at most $l$ successful outcomes. As
$\lim_{n\to\infty}F(n\cdot p';n,p)=0$ for all $p'<p$,
and as $\lim_{k\to\infty}|I_k|=\infty$, 
we can find some $k$ such that 
\begin{equation}\label{eq:bladw}
F(|I_k| p_j';|I_k|,p_j)<\frac1{2\cdot j^*}
\end{equation}
for all $j<j^*$, where we set 
$p_j'\defeq 1-\sqrt{\loss^{\bar p^j,n^j}}$ and 
$p_j\defeq 1-\frac{1+\sqrt2}2\cdot \loss^{\bar p^j,n^j}$.
(Note that $p_j'<p_j$, as $\loss^{\bar p^j,n^j}\le\frac12$.)
\item All elements of $Y=\{\alpha^{\bar p^j,n^j}_\ell:\ j<j^*\text{ and }\ell\in I_k\}$ 
are larger than $\beta$.
(This is possible as each sequence $(\alpha^{\bar p^j,n^j}_\ell)_{\ell<\omega}$
has supremum $\alpha$.)
We enumerate $Y$ by the
increasing sequence $(\beta_i)_{i\in M}$, and 
set $\beta_{-1}=\beta$.
\end{itemize}

We will find $q\le p^*$ forcing that $k\in \bigcap_{j<j^*}  A_{\bar p^j,n^j}$. 
To this end, we define a finite tree $\mathcal T$ of height $M$,
and assign to each $s\in \mathcal T$ of height $i$ 
a condition $q_s\in P_{\beta_{i-1}+1}$ (decreasing along each branch) 
and a probability $\prob_s\in[0,1]$,
such that $\sum_{t\rhd s}\prob_{t}=1$ for all non-terminal nodes $s\in \mathcal T$.
For $s$ the root of $\mathcal T$, i.e., for the unique $s$ of height 0, we set 
$q_{s}=p^*\in P_{\beta_{-1}}$ and $\prob_{s}=1$.


So assume we have already constructed $q_s\in P_{\beta_{i-1}+1}$ for some $s$ of height $i<M$.
We will now
take care of index $\beta_i$ and construct the set of successors of $s$, and for each successor $t$,
a $q_t\le q_s$ in $P_{\beta_{i}+1}$.


\begin{itemize}
\item If $\beta_i\in S^0$, the guardrail guarantees that 
$\beta_i\in\supp(p^j_\ell)$ implies
$p^j_\ell\restriction\beta_i\Vdash i_{\beta_i}(p^j_\ell(\beta_i))=h(\beta_i)$. 
In that case we use a unique $\mathcal T$-successor $t$ of $s$,
and we set
$q_t=q_s^\frown (\beta_i,i^{-1}_{\beta_i}h(\beta_i))$, and $\prob_t=1$.

In the following we assume $\beta_i\notin S^0$.
\item 
Let $J_i$  
be the set of $j<j^*$ such that there is an $\ell\in I_k$
with $\alpha^{\bar p^j,n^j}_\ell=\beta_i$
(there is at most one such $\ell$).
For $j\in J_i$,
set $r_i^j=p^j_\ell(\beta_i)$ for the according $\ell$.
So each $r_i^j$ is a $P_{\beta_i}$-name for an element of $Q_{\beta_i}$.

The guardrail gives us the constant value $(\stem^*_i,\loss^*_i)\defeq h(\beta_i)$
(which is equal to $(\stem^{\bar p^j,n^j},\loss^{\bar p^j,n^j})$ for all $j\in J_i$).

\item
The case $\beta_i\in S^3$, i.e., the case of random forcing, is basically~\cite[2.14]{sh}:

For $x\subseteq [\stem^*_i]$, set $\Leb^\rel(x)=\frac{\Leb(x)}{\Leb([\stem^*_i])}$.
Note that the
$r_i^j$ are closed subsets of $[\stem^*_i]$ and
$\Leb^\rel(r_i^j)\ge 1-\loss^*_i$.

Let $\mathcal B^*$ be the power set of $[\stem^*_i]$; and 
let $\mathcal B$ be the sub-Boolean-algebra generated by 
by $r_i^j$ ($j\in J_i$),  let
$\mathcal X$ be the set of atoms and 
$\mathcal X'=\{x\in\mathcal X:\, \Leb^\rel(x)>0\}$.
So $|\mathcal X'|\le 2^{J_i}\le 2^{j^*}$,
$\sum_{x\in \mathcal X'} \Leb^\rel(x)=1$,
and $\sum_{x\in \mathcal X', x\subseteq r_i^j} \Leb^\rel(x)=\Leb^\rel(r_i^j)$.

So far, $\mathcal X'$ is a $P_{\beta_i}$-name.
Now we increase $q_s$ inside $P_{\beta_i}$ to some $q^+$ deciding
which of the (finitely many) Boolean combinations result in elements of $\mathcal X'$, and also deciding
rational numbers $y_x$ ($x\in \mathcal X'$)
with sum $1$ such that
$|\Leb^\rel(x)-y_x|<\frac{\sqrt2-1}{2}\cdot\loss^*_i\cdot 2^{-j^*}$.

We can now define the immediate successors of $s$ in $\mathcal T$:
For each $x\in \mathcal X'$, add an immediate successor $t_x$ 
and assign to it the probability $\prob_{t_x}=y_x$ and 
the condition $q_{t_x}={q^+}^\frown (\beta_i,r_x)$, where $r_x$ is a (name for a) partial random condition below $x$ (such a condition exists, as the Lebesgue positive intersection
of finitely many partial random condition contains a partial random condition).

Note that when we choose a successor $t$ randomly (according to the
assigned probabilities $\prob_t$), 
then for each $j\in J$ 
the probability of $q^+\Vdash q_t(\beta_i)\le r_i^j$ is 
at least 
\begin{multline*}
\textstyle
\sum_{x\in \mathcal X',x\subseteq r^j_i}\prob_x
\ge 
\sum_{x\in \mathcal X',x\subseteq r^j_i}\big(\Leb^\rel(x)-\frac{\sqrt2-1}{2}\cdot\loss^*_i\cdot 2^{-j^*}\big)\ge\\
\textstyle
\ge 
\big(\sum_{x\in \mathcal X',x\subseteq r^j_i} \Leb^\rel(x)\big)-\frac{\sqrt2-1}{2}\cdot\loss^*_i
=
\Leb^\rel(r^j_i)-\frac{\sqrt2-1}{2}\cdot\loss^*_i\ge\\
\textstyle
\ge 
1-\loss^*_i-\frac{\sqrt2-1}{2}\cdot\loss^*_i
= 
1-\frac{1+\sqrt2}{2}\cdot\loss^*_i.
\end{multline*}
\item
The case $\beta_i\in S^4$, i.e., the case of $\QII$:

Recall that $\QII$-conditions are subtrees of some basic compact tree $T^*$,
and there is a $h$ such that: if $\max\{|I_k|, j^*\}$ many conditions share a common node (above their stems) at height $h$, then they are compatible.

All conditions $r_i^j$ have the same stem $s^*=\stem^*_i$.
For each $j\in J_i$, set $d(j)=r_i^j\cap \omega^{h}$. 
Note that ($P_{\beta_i}$ forces that) $d(j)$ is a subset of $T^*\cap [s^*]\cap \omega^h$ 
of relative size $\ge 1-\frac12\loss^*_i$ (according to Lemma~\ref{lem:QIIbasic}(\ref{item:ghgh})).
First find $q^+\le q_s$ in $P_{\beta_i}$ deciding all $d(j)$.

We can now define the immediate successors of $s$ in $\mathcal T$:
For each $x\in T^*\cap [s^*]\cap \omega^h$ add an immediate successor  
$t_x$, and assign to it the uniform probability (i.e., $\prob_{t_x}=\frac{1}{|T^*\cap [s^*]\cap \omega^h|}$) and the condition $q_{t_x}={q^+}^\frown(\beta_i,r_x)$, where 
$r_x$ is a partial $\QII$-condition stronger than all $r_i^j$ that satisfy $x\in d(j)$. 
(Such a condition exists, as we can intersect $\le j^*$ many conditions 
of height $h$.)

If we chose $t$ randomly, then for each $j\in J$ the 
probability of $q^+\Vdash q_t\le r_i^j$ is at least $1 - \frac12\loss^*_i  \ge 1-\frac{1+\sqrt2}2\cdot \loss^*_i$.
\end{itemize}

In the end, we get a tree $\mathcal T$ of height $M$, and we can chose a 
random branch through $\mathcal T$, according to the assigned 
probabilities. We can identify the branch with its terminal node $t^*$, so in this notation the branch $t^*$  has probability $\prod_{n\le M}\prob_{t^* \restriction n}$.

Fix $j<j^*$. There are 
$|I_k|$ many levels $i<M$ such that at $\beta_i$ we deal with 
the $(\bar p^j,n^j)$-case. Let $M^j$ be the set of these levels.
For each $i\in M^j$,
we perform an experiment, by asking whether the next step $t\in \mathcal T$
(from the current $s$ at level $i$) will satisfy
$q_t\restriction \beta_i\forces q_t(\beta_i)\le r_i^j$.
While the exact probability for success will depend on which $s$
at level $i$ we start from,
a lower bound is given by $1-\frac{1+\sqrt2}2\cdot \loss^*_i$. 
Recall that $\loss^*_i=\loss^{\bar p^j,n^j}$, and that we set 
$p_j\defeq 1-\frac{1+\sqrt2}2\cdot \loss^*_i$ and 
$p_j'\defeq 1-\sqrt{\loss^{\bar p^j,n^j}}$ 
 in~\eqref{eq:bladw}.
So the chance of our branch $t^*$
having success fewer than $|I_k|\cdot (1-\sqrt{\loss^{\bar p^j,n^j}})$ many times,
out of the the $|I_k|$ many tries,
(let us call such a $t^*$ ``bad for $j$'')
is at most $F(|I_k| p';|I_k|,p)\le \frac1{2 j^*}$.

Accordingly, the measure of branches that are not bad for \emph{any} $j<j^*$
is at least $\frac12$. Fix such a branch $t^*$.
Then 
for each $j<j^*$,
\[
\big|\big\{i\in M^j:\, q_{t^*}\restriction \beta_i\Vdash q_{t^*}(\beta_i)\le r^j_i\big\}\big|\ge |I_k|\cdot \Big(1-\sqrt{\loss^{\bar p^j,n^j}}\Big),
\]
and thus $q_{t^*}$ forces that
\[
\big|\big\{\ell\in I_k:\,  p_\ell(\alpha^{\bar p^j,n^j}_\ell)\in G(\alpha^{\bar p^j,n^j}_\ell)\big\}\big|\ge |I_k|\cdot \Big(1-\sqrt{\loss^{\bar p^j,n^j}}\Big).
\]
\textbf{\boldmath(c) for $\cf(\alpha)=\omega$:}

Fix $\bar p$ as in the assumption of (c).
To simplify notation, let us assume that $\nabla\ne \emptyset$ and that 
$\sup(\nabla)<\sup(\supp(p_\ell))$
(for some, or equivalently: all,
$\ell\in\omega$). Let $0<n_0<m^{\bar p}$ be such that $\sup(\nabla)$ is
at position $n_0-1$ in
$\supp(p_\ell)$, i.e., $\sup(\nabla)=\alpha^{\bar p,n_0-1}_\ell$  
(independent of $\ell$), and set 
$\beta\defeq \sup(\nabla)+1$.

$\bar p\restriction\beta$ is again a countable $\Delta$-system following
the same $h$, and 
$\lim_{(\Xi^h_\gamma)_{\gamma<\alpha}}(\bar p)$ is by definition identical
to $\lim_{(\Xi^h_\gamma)_{\gamma<\beta}}(\bar p\restriction \beta)$,
which by induction is a valid condition and forces 
(c) for $\bar p\restriction \beta$. This gives us the set $A_{\bar p\restriction \beta}$ of measure at least $1-\sum_{n<n_0} \sqrt{\loss^{\bar p,n}}$.

For the positions $n_0\le n<m^{\bar p}$, 
all $(\alpha^{\bar p,n}_\ell)_{\ell\in\omega}$ are strictly increasing
sequences above $\beta$ with some limit $\alpha_n\le\alpha$.
Then (b) (applied to $\alpha_n$) gives us an according measure-1-set 
$A_{\bar p,n}$.

So $\lim_{(\Xi^h_\gamma)_{\gamma<\alpha}}(\bar p)$
forces that  $A'=A_{\bar p\restriction \beta}\cap \bigcap_{n_0\le n<m^{\bar p}}
A_{\bar p,n}$ has measure  $\Xi^h_\alpha(A')\ge 1-\sum_{n<n_0} \sqrt{\loss^{\bar p,n}}\ge 1-\sum_{n<m^{\bar p}} \sqrt{\loss^{\bar p,n}}$. 

Note that
$p_\ell\in G$ iff $p_\ell \restriction \beta\in G_{\beta}$ and
$p_\ell(\alpha^{\bar p,n})\in G(\alpha^{\bar p,n})$ for all $n_0\le n<m^{\bar p}$. 

Fix $k\in A'$. 
As $k\in A_{\bar p\restriction \beta}$, 
the relative frequency for $\ell\in I_k$ to \emph{not} satisfy
$p_\ell \restriction \beta\in G_{\beta}$ is at most $\sum_{n<n_0}\sqrt{\loss^{\bar p,n}}$. For any $n_0\le n<m^{\bar p}$, as $k\in A_{\bar p,n}$, the relative frequency for
\emph{not} $p_\ell(\alpha^{\bar p,n})\in G(\alpha^{\bar p,n})$ is at most
$\sqrt{\loss^{\bar p,n}}$. So the relative frequency for 
$p_\ell\in G$ to fail is at most $\sum_{n<n_0}\sqrt{\loss^{\bar p,n}}+\sum_{n_0\le n<m^{\bar p}}\sqrt{\loss^{\bar p,n}}$, as required.

\textbf{\boldmath(a\&c) for $\alpha=\gamma+1$ successor:}

For $\gamma\in S^0$ this is clear: Let $\Xi^h_\alpha$ be the name of
some FAM extending $\Xi^h_\gamma$. 
Let $\bar p$ be as in (c), without loss of generality
$\gamma\in\nabla$. 
Then $q^+\defeq \lim_{(\Xi^h_\beta)_{\beta<\alpha}}(\bar p)=q^\frown (\gamma,r)$,
where $q\defeq \lim_{(\Xi^h_\beta)_{\beta<\gamma}}(\bar p\restriction \gamma)$
and $r$ is the condition determined by $h(\gamma)$,
i.e., each $p_\ell\restriction\gamma$ forces $p_\ell(\gamma)=r$.
In particular, $q^+$ forces that
$p_\ell\in G_\alpha$ iff $p_\ell\restriction \gamma\in G_\alpha$.
By induction, (c) holds for $\gamma$, and therefore we get (c) for $\alpha$.

Assume $\gamma\in S^3\cup S^4$. 
By induction we know that (d) holds for $\gamma$,
i.e., that $\Xi^h_\gamma$ restricted to $M_\gamma$ (call it $\Xi_0$)
is in $M_\gamma$.
So the requirement in the definition~\ref{def:Mlimit} of the limit 
is satisfied, and thus the limit $q^+\defeq \lim_{\bar \Xi^h}(\bar p)$
is well defined
for any 
countable $\Delta$-system $\bar p$ as in (c): 
$q^+$ has the form $q^\frown(\gamma, r)$
with 
$q=\lim_{(\Xi^h_\beta)_\beta<\gamma}(\bar p\restriction \gamma)$
and $r=\lim_{\Xi_0}((p_\ell(\gamma))_{\ell\in\omega})$.
Now Lemma~\ref{lem:explicit} gives us the $P_\alpha$-name $\Xi^+$, 
which will be our new $\Xi^h_\alpha$.

This works as required: Again without loss of generality we can assume $\gamma\in\nabla$.
By induction,
$q$ forces that $\Xi^h_\gamma(A_{\bar p\restriction \gamma})\ge 1-\sum_{n<m^p-1}\sqrt{\loss^{\bar p,n}}$.
According to Lemma~\ref{lem:explicit},
$r$ forces that $\Xi^+(A_{(p_\ell(\gamma))_{\ell\in\omega}})\ge 1-\sqrt{\loss^{\bar p,m^p-1}}$.
So $q^+=q^\frown r$ forces that $\Xi^h_\alpha(A_{\bar p})\ge 1-\sum_{n<m^p}\sqrt{\loss^{\bar p,n}}$.
%
%

\textbf{(d):}

So we have (in $V$) the $P_\alpha$-name $\Xi^h_\alpha$.
We already know  that there is (in $V$)
an $\omega_1$-club set $X_0$ in $[\alpha]^{<\lambda_i}$
(for the appropriate $i\in\{3,4\}$)
such that $w\in X_0$ implies that 
$w$ satisfies Assumptions~\ref{asm:P} and~$\ref{asm:completesubforcing}$.
So each such $w\in X_0$ defines a
complete subforcing $P_w$ of $P_\alpha$
and the $P_\alpha$-mame for the according $P_w$-extention
$M_w$. 

Fix some $w\in X_0$. We will define $w'\supseteq w$ as follows:
For a $P_w$-name (and thus a $P_\alpha$-name)
$r\in 2^\omega$, let $s$ be the name of $\Xi_\alpha(r)\in[0,1]$.
As in Lemma~\ref{lem:closure1}(\ref{item:borel}), we can find a countable $w_r$
determining $s$. (I.e., there is a Borel function that calculates the real $s$ from the generics at $w_r$; moreover we know this Borel function in the ground model.) Let $w'\supseteq w$ be in $X_0$ 
and contain all these $w_r$,
for a (small representative set of) all $P_w$-names for reals.

Iterating this construction $\omega_1$ many steps gives 
us a suitable $w_\alpha$:
Note that the assignment of a name $r$ to the $\Xi_\alpha$-value $s$
can be done in $V$, and thus is known to $M_\alpha$.
In addition, $M_\alpha$ sees that for each ``actual real''
(i.e., element of $M_\alpha$), 
the value $s$ is already determined (by $P_\alpha'$).
So the assignment $r\mapsto s$, which is $\Xi_\alpha$
restricted to $M_\alpha$, is in $M_\alpha$.
\end{proof}



Note that in (c), when we deal with a countable $\Delta$-system $\bar p$ following the guardrail $h\in H^*$, the condition $\lim_{\bar \Xi^h}\bar p$
 forces in particular that infinitely many $p_\ell$ are in $G$.
So after carrying out the construction as above, we get a forcing
notion $\Pa$ satisfying the following (which is actually the only 
thing we need from the previous construction, in addition
to the fact that we can choose each $w_\alpha$ in an $\omega_1$-club):
\begin{lemma}\label{lem:deltabla}
For every countable $\Delta$-system $\bar p$
there is some $q$ forcing that infinitely many $p_\ell$
are in the generic filter.
\end{lemma}
\begin{proof}
According to Lemma~\ref{use.EK}, $\bar p$ follows some $h\in H^*$;
so $q=\lim_{\bar \Xi^h}(\bar p)$ will work.
\end{proof}

\begin{lemma}\label{lem:bissmall}
$\mylinII(\Pa,\kappa)$ for $\kappa\in[\lambda_2,\lambda_5]$ regular,
witnessed by the sequence $(c_\alpha)_{\alpha<\kappa}$ of the first $\kappa$ many Cohen reals.
\end{lemma}

\begin{proof}
Fix a $\Pa$-name $y\in \omega^\omega$. We have to show that
$(\exists \alpha\in\kappa)\, (\forall \beta\in \kappa\setminus \alpha)\,\Pa\forces \lnot c_\beta \le^* y)$.

Assume towards a contradiction that $p^*$ forces that 
there are unboundedly many $\alpha\in\kappa$
with $c_\alpha\le^* y$, and enumerate them as $(\alpha_i)_{i\in\kappa}$.
Pick $p^i\le p^*$ deciding $\alpha_i$ to be some $\beta^i$, and also deciding 
$n_i$ such that $(\forall m\ge n_i)\, c_{\alpha_i}(m)\le y(m)$.
We can assume that $\beta^i\in\supp(p^i)$.
Note that $\beta^i$ is a Cohen position (as $\beta^i<\kappa\le \lambda_5$), and we can assume that $p^i(\beta^i)$ is a Cohen condition in $V$ (and not just a $P_{\beta_i}$-name for such a 
condition).
By strengthening and thinning out, we may assume:
\begin{itemize}
\item $(p^i)_{i\in\kappa}$ forms a 
$\Delta$ system with heart $\nabla$.
\item
All $n_i$ are equal to 
some $n^*$.
 \item $p^i(\beta_i)$ is always the same Cohen 
condition $s\in\omega^{<\omega}$, 
without loss of generality of length $|s|=n^{**}\ge n^*$.
\item For some position $n<m^{\bar p}$,
 $\beta^i$ is the $n$-th element of $\supp(p^i)$.
\end{itemize}
Note that this $n$ cannot be a heart condition:
For any $\beta\in\kappa$, at most 
$|\beta|$ many $p^i$ can force $\alpha_i=\beta$, as 
$p^i$ forces that $\alpha_i\ge i$ for all $i$.

Pick a countable subset of this $\Delta$-system which 
forms a countable $\Delta$-system $\bar p\defeq (p_\ell)_{\ell\in\omega}$.
So $p_\ell=p^{i_\ell}$ for some $i_\ell\in \kappa$,
and we set $\beta_\ell=\beta^{i_\ell}$. In particular all
$\beta_\ell$ are distinct.
Now extend each $p_\ell$ to $p'_\ell$ by extending
the Cohen condition $p_\ell(\beta_\ell)=s$ to
$s^\frown \ell$ (i.e., forcing $c_{\beta_\ell}(n^{**})=\ell$).
Note that 
$\bar p'\defeq(p'_i)_{i\in\omega}$ is still a countable $\Delta$-system,%
\footnote{Note that $\bar p'$ will not follow the same guardrail as $\bar p$.}
and by Lemma~\ref{lem:deltabla}
some $q$ forces that infinitely many of the
$p'_\ell$ are in the generic filter.
But each such $p'_\ell$ forces that $c_{\beta_\ell}(n^{**})=\ell \le y(n^{**})$,  a contradiction.
\end{proof}

\subsection{The left hand side}
We have now finished the consistency proof for the left hand side:
\begin{theorem}~\label{thm:left}
Assume GCH and let $\lambda_i$ be an increasing sequence of regular cardinals, none of which is a successor of a cardinal of countable cofinality
for $i=1,\dots,5$.
Then there is a cofinalities-preserving forcing $P$ resulting in
\begin{multline*}
\addnull=\lambda_1  < \addmeager= \mathfrak{b} =\lambda_2< \covnull =\lambda_3<
\nonmeager =\lambda_4<\\< \covmeager = 2^{\aleph_0}=\lambda_5.
\end{multline*}
\end{theorem}
\begin{proof}
Set $\chi=\lambda_2$, and 
let $R$ be the set of partial functions $f:\chi\times \lambda_5\to 2$ with $|\dom(f)|<\chi$ (ordered by inclusion).
$R$ is ${<}\chi$-closed, $\chi^+$-cc, 
and adds $\lambda_5$ many new elements to $2^\chi$.
So in 
the $R$-extension, Assumption~\ref{asm:chi} is satisfied, and we can 
construct $\Pa$ according to Assumption~\ref{asm:P} 
and Construction~\ref{constr}. Fact~\ref{fact:summary} gives us all inequalities for the left hand side, apart from $\mathfrak{b}\le\lambda_2$, which we get from~\ref{lem:bissmall}.

In the $R$-extension, CH holds and $P$ is a FS ccc iteration of length $\delta_5$,
 $|\delta_5|=\lambda_5$, and each iterand is a set of reals; so $2^{\aleph_0}\le \lambda_5$ is forced. Also, any FS ccc iteration of length $\delta$ (of nontrivial iterands)
forces 
$\covmeager\ge\cf(\delta)$: Without loss of generality $\cf(\delta)=\lambda$ is uncountable.
Any set $A$ of (Borel codes for) meager sets 
that has size ${<}\lambda$ already appears at some 
stage $\alpha<\delta$, and the iteration at state $\alpha+\omega$ adds a Cohen real
over the $V_\alpha$, so $A$ will not cover all reals.
\end{proof}
\begin{remark}
So this consistency result is reasonably general,
we can, e.g., use the values $\lambda_i=\aleph_{i+1}$
.
This is in contrast to the result for the whole diagram, 
where in particular the
small $\lambda_i$ have to be separated by strongly compact cardinals.
\end{remark}

\section{Ten different values in Cicho\'n's diagram}\label{sec:partB}

We can now apply, with hardly any change, the technique of~\cite{ten}
to get the following:

\begin{theorem}\label{thm:ten}
Assume GCH and that $\aleph_1<\kappa_9<\lambda_1<\kappa_8<\lambda_2<\kappa_7<\lambda_3<\kappa_6<\lambda_4<\lambda_5<\lambda_6<\lambda_7<\lambda_8<\lambda_9$
are regular,  $\lambda_i$ is not
a successor of a cardinal of countable cofinality
for $i=1,\dots,5$, 
$\lambda_2=\chi^+$ with $\chi$ regular, and $\kappa_i$ strongly compact for $i=6,7,8,9$.
Then there is a ccc forcing notion $\PaIX$ resulting in: 
\begin{multline*} 
\addnull=\lambda_1 <  \mathfrak{b}=\addmeager=\lambda_2<\covnull =\lambda_3<\nonmeager=\lambda_4<\\ <\covmeager=\lambda_5 < \nonnull=\lambda_6 <  \mathfrak{d}=\cofmeager= \lambda_7 <\cofnull=\lambda_8< 2^{\aleph_0}=\lambda_9.\end{multline*}
\end{theorem}

To do this, we first have to show that we can achieve the 
order for the left hand side, i.e., Theorem~\ref{thm:left},
starting with GCH and using a FS ccc iteration $\Pa$ alone
(instead of using $P=R* \Pa$, where $R$ is not ccc).
This is the only argument that requires $\lambda_2=\chi^+$. 
We will just briefly sketch it here, as it can be found with all details in~\cite[1.4]{ten}:

\begin{itemize}
\item We already know that in the $R$-extension,
(where $R$ is ${<}\chi$-closed, $\chi^+$-cc and
forces $2^{\chi}=\lambda_5$) we can find by the
inductive construction~\ref{constr} suitable $w_\alpha$
such that $R*\Pa$ works.
\item 
We now perform a similar inductive construction in the ground model: At stage $\alpha$,
we know that there is an $R$-name for a suitable $w_\alpha^1$
of size $<\lambda_i$ (where $i$ is $3$ in the random and $4$ in the $\QII$-case).
This name can be covered by some set $\tilde w_\alpha^1$ in $V$, 
still of size $<\lambda_i$, as $R$ is $\chi^+$-cc.
Moreover, in the $R$-extension, 
the suitable parameters form  an $\omega_1$-club;
so there is a suitable  $w_\alpha^2\supseteq \tilde w_\alpha^1$, etc.
Iterating $\omega_1$ many times and taking the union at the end
leads to $w_\alpha$ in $V$ which is forced by $R$ to be suitable.
\item Not only $w_\alpha$ is in $V$, but the construction for $w_\alpha$
is performed in $V$, so we can construct the 
whole sequence $\bar w=(w_\alpha)_{\alpha\in\delta_5}$ in $V$.
\item We now know that in the $R$-extension, the forcing 
$\Pa$ defined from $\bar w$ will satisfy $\mylinII(\Pa,\kappa)$ in the form of Lemma~\ref{lem:bissmall}.
\item By an absoluteness argument, we can show that actually 
in $V$ the forcing 
$\Pa$ defined form $\bar w$ will satisfy~Lemma~\ref{lem:bissmall}
as well.
\end{itemize}

The rest of the proof is the same as in
\cite[Sec.~2]{ten}, where we interchange $\mathfrak b$ and
$\covnull$ as well as $\mathfrak d$ and $\nonnull$.

We cite the following facts from~\cite[2.2--2.5]{ten}:
\begin{facts}\label{facts:ten}
\begin{enumerate}[(a)]
\item\label{item:first}  If $\kappa$ is a strongly compact cardinal and $\theta>\kappa$
regular, then there is an elementary embedding $j_{\kappa,\theta}:V\to M$
(in the following just called $j$) such that
\begin{itemize}
\item the critical point of $j$ is $\kappa$,
 $\cf(j(\kappa))=|j(\kappa)|=\theta$, 
\item $\max(\theta,\lambda)\le j(\lambda)< \max(\theta,\lambda)^+$
for all $\lambda\ge \kappa$ regular, and
\item $\cf(j(\lambda))=\lambda$ for $\lambda\ne\kappa$
regular,
\end{itemize}
and such that the following is satisfied:
\item\label{item:presiter} If $P$ is a FS ccc iteration along $\delta$,
then $j(P)$ is a FS ccc iteration along $j(\delta)$.
\item\label{item:lin}  
$\mylini(P,\lambda)$ implies  $\mylini(j(P),\cf(j(\lambda)))$, and thus
$\mylini(j(P),\lambda)$ if $\lambda\neq \kappa$ regular.\footnote{In~\cite{ten}, we only used ``classical'' relations $\RIII$ that are defined on a Polish space in an
absolute way. In this paper, we use the relation $\RIII$ which is not of this kind.
However, the proof still works without any change:
The parameter $\mathcal E$ used to define the relation $\RIII$, cf.\ Definition~\ref{def:mathcalE}, is a set of reals. So
$j(\mathcal E)=\mathcal E$, and we can still use the usual absoluteness arguments between 
$M$ and $V$. (A parameter not element of $H(\kappa_9)$ 
might be a problem.)}
\item\label{item:part}\label{item:last} 
If $\myparti(P,\lambda,\mu)$, then
$\myparti(j(P),\lambda,\mu')$, for
$\mu'=\begin{cases}
|j(\mu)|    &    \text{if }\kappa> \lambda 
\\
\mu &\text{if }\kappa< \lambda.
\end{cases}
$
\end{enumerate}
\end{facts}

Using these facts, it is easy to finish the proof:\footnote{This is identical to the argument in~\cite{ten}, with the roles of $\mathfrak{b}$
and $\covnull$, as well as their duals,  switched.}
\begin{proof}[Proof of Theorem~\ref{thm:ten}]
Recall that we want to force the following values to
the characteristics of Figure~\ref{fig:ten} (where we indicate the positions of the $\kappa_i$ as well):
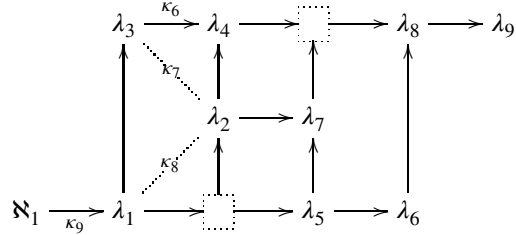
\[
\xymatrix@=4.5ex{
&            \lambda_3\ar[r]^{\kappa_6}        & \lambda_4 \ar[r]      &  \mye \ar[r]     & \lambda_8\ar[r] &\lambda_9 \\
&                               & \lambda_2\ar[r]\ar[u]\ar@{.}[lu]|-{\kappa_7}  &  \lambda_7\ar[u] &              \\ 
  \aleph_1\ar[r]_{\kappa_9} & \lambda_1\ar[r]\ar[uu]\ar@{.}[ru]|-{\kappa_8} & \mye\ar[r]\ar[u] &  \lambda_5\ar[r]\ar[u]& \lambda_6\ar[uu]  
}
\]       
%

\textbf{Step 5:} Our first step, called ``Step 5'' for notational reasons, just uses $\Pa$. This is an iteration of length $\delta_5$ with $\cf(\delta_5)=|\delta_5|=\lambda_5$, satisfying: 
\begin{equation}\label{eq:tmpV}
\parbox{0.8\columnwidth}{%
For all $i$:
$\mylini(\Pa,\mu)$ for all $\mu\in[\lambda_i,\lambda_5]$ regular, 
and $\myparti(\Pa,\lambda_i,\lambda_5)$.
}
\end{equation}
As a consequence, the characteristics 
are forced by $\Pa$ to have the following values\footnote{These values,
and the ones forced by the ``intermediate forcings'' $\PaVI$ to $\PaVIII$,
are  not required for the argument; they should just illustrate what is going on.} (we also mark the 
position of $\kappa_6$, which we are going to use in 
the following step):
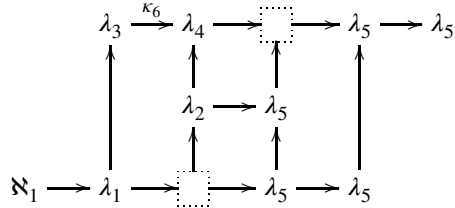
\[
\xymatrix@=3.5ex{
&            \lambda_3\ar[r]^{\kappa_6}        & \lambda_4 \ar[r]      &  \mye \ar[r]     & \lambda_5\ar[r] &\lambda_5 \\
&                               & \lambda_2\ar[r]\ar[u]  &  \lambda_5\ar[u] &              \\ 
  \aleph_1\ar[r] & \lambda_1\ar[r]\ar[uu] & \mye\ar[r]\ar[u] &  \lambda_5\ar[r]\ar[u]& \lambda_5\ar[uu]  
}
\]       
%

\textbf{Step 6:}
Consider the embedding $j_6\defeq j_{\kappa_6,\lambda_6}$.
According to~\ref{facts:ten}(\ref{item:presiter}),
$\PaVI\defeq j_6(\Pa)$ is a FS ccc iteration of length $\delta_6\defeq j_6(\delta_5)$.
As $|\delta_6|=\lambda_6$, the continuum is forced to have size $\lambda_6$.

For $i=1$, we have
$\mylinI(\Pa,\mu)$ for all regular 
$\mu\in[\lambda_1,\lambda_5]$, so using~\ref{facts:ten}(\ref{item:lin}) we get 
$\mylinI(\PaVI,\mu)$ for all regular 
$\mu\in[\lambda_1,\lambda_5]$ different to $\kappa_6$;
as well as $\mylinI(\PaVI,\lambda_6)$ (as $\cf(j(\kappa_6))=\lambda_6$).
For $\mu=\lambda_1$ the former implies 
$\PaVI\Vdash \addnull\le \lambda_1$, and the latter 
$\PaVI\Vdash \cofnull\ge \lambda_6=2^{\aleph_0}$. 

More generally, we get from~\eqref{eq:tmpV} and~\ref{facts:ten}(\ref{item:lin})
\begin{equation}\label{eq:tmpVIl}
\parbox{0.8\columnwidth}{%
For all $i$: $\mylini(\PaVI,\mu)$
for all regular 
$\mu\in[\lambda_i,\lambda_5]\setminus \{\kappa_6\}$.
\\
For $i<4$:
$\mylini(\PaVI,\lambda_6)$.
}
\end{equation}

So in particular for $\mu=\lambda_i$, we see that 
the characteristics on the left do not increase;
for $\mu=\lambda_5$ that the ones on the right are still at least $\lambda_5$;
and for $i<4$ an $\mu=\lambda_6$ that
the according characteristics
on the right will have size continuum. (But not for $i=4$,
as $\kappa_4<\lambda_4$. And we will see that $\covmeager$
is at most $\lambda_5$.)

Dually, because $\lambda_3<\kappa_6<\lambda_4$, we get
from~\eqref{eq:tmpV} and~\ref{facts:ten}(\ref{item:part})
\begin{equation}\label{eq:tmpVIp}
\parbox{0.8\columnwidth}{%
For $i<4$: $\myparti(\PaVI,\lambda_i,\lambda_6)$. For $i=4$:
$\mypartIV(\PaVI,\lambda_4,\lambda_5)$.
}
\end{equation}
(The former because $|j_6(\lambda_5)|=\max(\lambda_6,\lambda_5)=\lambda_6$.)
So the characteristics on the left do not decrease, 
and $\PaVI\Vdash \covmeager \le \lambda_5$.

Accordingly, $\PaVI$ forces the following values:
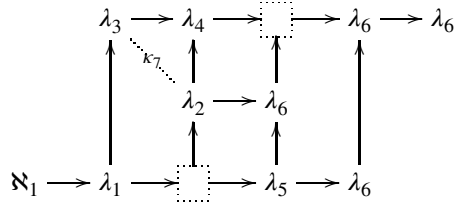
\[
\xymatrix@=3.5ex{
&            \lambda_3\ar[r]        & \lambda_4 \ar[r]      &  \mye \ar[r]     & \lambda_6\ar[r] &\lambda_6 \\
&                               & \lambda_2\ar[r]\ar[u]\ar@{.}[lu]|-{\kappa_7}  &  \lambda_6\ar[u] &              \\ 
  \aleph_1\ar[r] & \lambda_1\ar[r]\ar[uu] & \mye\ar[r]\ar[u] &  \lambda_5\ar[r]\ar[u]& \lambda_6\ar[uu]  
}
\] 
%


\textbf{Step 7:}
We now apply a new embedding, $j_7\defeq j_{\kappa_7,\lambda_7}$,
to the forcing $\PaVI$ that 
we just constructed. (We always work in $V$, not in any inner model $M$ or any forcing
extention.)
As before,
set $\PaVII\defeq j_7(\PaVI)$, a FS ccc iteration of length $\delta_7=j_7(\delta_6)$,
forcing the continuum to have size $\lambda_7$.

Now $\kappa_7\in(\lambda_2,\lambda_3)$, so arguing as before, we get from~\eqref{eq:tmpVIl}
\begin{equation}\label{eq:tmpVIIl}
\parbox{0.8\columnwidth}{%
For all $i$: $\mylini(\PaVII,\mu)$
for all regular 
$\mu\in[\lambda_i,\lambda_5]\setminus \{\kappa_6,\kappa_7\}$.
\\
For $i<4$: 
$\mylini(\PaVII,\lambda_6)$. For $i<3$: $\mylini(\PaVII,\lambda_7)$.
}
\end{equation}
and from~\eqref{eq:tmpVIp}
\begin{equation}\label{eq:tmpVIIp}
\parbox{0.8\columnwidth}{%
For $i<3$:
$\myparti(\PaVII,\lambda_i,\lambda_7)$.
\\
For $i=3$:
$\mypartIII(\PaVII,\lambda_3,\lambda_6)$.
For $i=4$:
$\mypartIV(\PaVII,\lambda_4,\lambda_5)$.
}
\end{equation}
Accordingly, $\PaVII$ forces the following values:
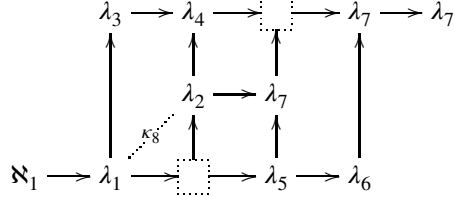
\[
\xymatrix@=3.5ex{
&            \lambda_3\ar[r]        & \lambda_4 \ar[r]      &  \mye \ar[r]     & \lambda_7\ar[r] &\lambda_7 \\
&                               & \lambda_2\ar[r]\ar[u]  &  \lambda_7\ar[u] &              \\ 
  \aleph_1\ar[r] & \lambda_1\ar[r]\ar[uu]\ar@{.}[ru]|-{\kappa_8} & \mye\ar[r]\ar[u] &  \lambda_5\ar[r]\ar[u]& \lambda_6\ar[uu]  
}
\]       
%

\textbf{Step 8:}
Now we set $\PaVIII\defeq j_{\kappa_8,\lambda_8}(\PaVII)$, a FS ccc iteration
of length $\delta_8$. 
Now $\kappa_8\in(\lambda_1,\lambda_2)$, and as before, we get from~\eqref{eq:tmpVIIl}
\begin{equation}\label{eq:tmpVIIIl}
\parbox{0.8\columnwidth}{%
For all $i$: $\mylini(\PaVIII,\mu)$
for all regular 
$\mu\in[\lambda_i,\lambda_5]\setminus \{\kappa_6,\kappa_7,\kappa_8\}$.
\\
For $i<4$:
$\mylini(\PaVIII,\lambda_6)$.
For $i<3$: $\mylini(\PaVIII,\lambda_7)$.
\\
For $i<2$ (i.e., $i=1$): $\mylinI(\PaVIII,\lambda_8)$.
}
\end{equation}
and from~\eqref{eq:tmpVIIp}
\begin{equation}\label{eq:tmpVIIIp}
\parbox{0.8\columnwidth}{%
For $i=1$:
$\mypartI(\PaVIII,\lambda_1,\lambda_8)$.
For $i=2$:
$\mypartII(\PaVIII,\lambda_2,\lambda_7)$.
\\
For $i=3$:
$\mypartIII(\PaVIII,\lambda_3,\lambda_6)$.
For $i=4$:
$\mypartIV(\PaVIII,\lambda_4,\lambda_5)$.
}
\end{equation}
Accordingly, $\PaVIII$ forces the following values:
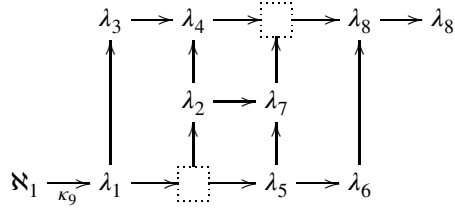
\[
\xymatrix@=3.5ex{
&            \lambda_3\ar[r]        & \lambda_4 \ar[r]      &  \mye \ar[r]     & \lambda_8\ar[r] &\lambda_8 \\
&                               & \lambda_2\ar[r]\ar[u]  &  \lambda_7\ar[u] &              \\ 
  \aleph_1\ar[r]_{\kappa_9} & \lambda_1\ar[r]\ar[uu] & \mye\ar[r]\ar[u] &  \lambda_5\ar[r]\ar[u]& \lambda_6\ar[uu]  
}
\]       
%

\textbf{Step 9:} Finally we set $\PaIX\defeq j_{\kappa_9,\lambda_9}(\PaVIII)$, a FS ccc iteration
of length $\delta_9$ with $|\delta_9|=\lambda_9$, i.e., the continuum will have 
size $\lambda_9$. As $\kappa_9<\lambda_1$, 
\eqref{eq:tmpVIIIl} and~\eqref{eq:tmpVIIIp} also hold for $\PaIX$ instead of $\PaVIII$.
Accordingly, we get the same values for the diagram as for $\PaVIII$, apart from the value for the continuum,
$\lambda_9$.
\end{proof}

\printbibliography

\end{document}